\documentclass[bj,english]{imsart_arxiv}
\RequirePackage[OT1]{fontenc}
\RequirePackage[colorlinks,citecolor=blue,urlcolor=blue]{hyperref}

\usepackage{float}
\usepackage{bm}
\usepackage{amsmath}
\usepackage{amssymb}
\usepackage{amsthm}

\usepackage{graphicx}
\usepackage{natbib}
\makeatletter

\floatstyle{ruled}
\newfloat{algorithm}{tbp}{loa}
\providecommand{\algorithmname}{Algorithm}
\floatname{algorithm}{\protect\algorithmname}

\@ifundefined{showcaptionsetup}{}{%
 \PassOptionsToPackage{caption=false}{subfig}}
\usepackage{subfig}
\makeatother

\startlocaldefs
\DeclareMathOperator*{\argmax}{argmax}
\DeclareMathOperator*{\argmin}{argmin}

\DeclareMathOperator*{\Tr}{Tr}
\numberwithin{equation}{section}
\theoremstyle{plain}
\newtheorem{thm}{Theorem}[section]
\newtheorem{lemma}{Lemma}[section]
\newtheorem{definition}{Definition}
\newtheorem{result}{Result}
\theoremstyle{remark}
\newtheorem{rem}{Remark}[section]

\global\long\def\ba{\bm{\alpha}}
\global\long\def\bb{\bm{\beta}}
\global\long\def\bt{\bm{\theta}}

\global\long\def\bz{\bm{z}}

\global\long\def\bl{\bm{\lambda}}
\global\long\def\ba{\bm{\alpha}}
\global\long\def\bb{\bm{\beta}}

\global\long\def\be{\bm{\eta}}

\global\long\def\I{\mathds{1}}

\global\long\def\R{\mathbb{R}}
\global\long\def\bL{\mathbf{L}}
\global\long\def\bM{\mathbf{M}}

\global\long\def\bI{\mathbf{I}}

\global\long\def\bU{\mathbf{U}}
\global\long\def\bR{\mathbf{R}}

\global\long\def\nus{\nu^{\star}}
\global\long\def\psib{\bar{\psi}}

\global\long\def\bK{\mathbf{K}}
\global\long\def\bY{\mathbf{Y}}

\global\long\def\diag{\mathrm{diag}}

\global\long\def\X{\mathcal{X}}

\global\long\def\Y{\mathcal{Y}}

\global\long\def\I{\mathbb{I}}

\newcommand{\vY}{\mathbf{Y}}

\newcommand\numberthis{\addtocounter{equation}{1}\tag{\theequation}}

\endlocaldefs

\usepackage{algorithm}
\usepackage{algpseudocode}
\usepackage{babel}
\begin{document}
\sloppy

\begin{frontmatter}

\title{Asymptotic Equivalence of Fixed-size and Varying-size Determinantal Point Processes}

\begin{aug}
  \author{\fnms{Simon} \snm{Barthelm\'e} \thanksref{a} \corref{} \ead[label=e1]{simon.barthelme@gipsa-lab.fr} }
  \author{ \fnms{Pierre-Olivier} \snm{Amblard} \thanksref{a} \ead[label=e2]{pierre-olivier.amblard@gipsa-lab.fr}  }
  \and
  \author{ \fnms{Nicolas} \snm{Tremblay} \thanksref{a}  \ead[label=e3]{nicolas.tremblay@gipsa-lab.fr}}
\address[a]{CNRS, Gipsa-lab, Grenoble INP and Universit\'e Grenoble Alpes. 
11 rue des Math\'ematiques 
Grenoble Campus
BP46
F - 38402 SAINT MARTIN D'HERES Cedex
FRANCE \\
\printead{e1} \printead{e2,e3} }
\runauthor{Barthelm\'e et al.}

\affiliation{CNRS, Gipsa-lab}

\end{aug}

\begin{abstract}
 \hspace{0.1cm} Determinantal Point Processes (DPPs) are popular models for point processes with repulsion. They appear in numerous contexts, from physics to graph theory, and display appealing theoretical properties. On the more practical side of things, since DPPs tend to select sets of points that are some distance apart (repulsion), they have been advocated as a way of producing random subsets with high diversity.
  DPPs come in two variants: fixed-size and varying-size. A sample from a varying-size DPP is a subset of random cardinality, while in fixed-size ``$k$-DPPs'' the cardinality is fixed. The latter makes more sense in many applications, but unfortunately their computational properties are less attractive, since, among other things, inclusion probabilities are harder to compute.
  In this work we show that as the size of the ground set grows, $k$-DPPs and DPPs become equivalent, in the sense that fixed-order inclusion probabilities converge. As a by-product, we obtain saddlepoint formulas for inclusion probabilities in $k$-DPPs. These turn out to be extremely accurate, and suffer less from numerical difficulties than exact methods do.
  Our results also suggest that $k$-DPPs and DPPs also have equivalent maximum likelihood estimators.  Finally, we obtain results on asymptotic approximations of  elementary symmetric polynomials which may be of independent interest. 
\end{abstract}


\maketitle

\end{frontmatter}

Determinantal Point Processes originally arose in quantum physics \citep{Macchi:CoincidenceApproach} and random matrix theory \citep{Soshnikov:DeterminantalRandomPointFields}, but they are such natural objects that they have also been rediscovered within computer science \citep{Deshpande:MatrixApproxVolSampling, Deshpande:EffVolSampl} and that special cases have appeared in the statistics literature as well \citep{Chen:WeightedFinitePopSampling}. Within Machine Learning, their current popularity owes much to \cite{KuleszaTaskar:DPPsforML}, whose overall approach we will mostly follow here. Like them, we focus on discrete DPPs.

\cite{KuleszaTaskar:DPPsforML} advocate DPPs as tractable probabilistic models for diverse subsets. Specifically, we assume that we have a ground set of $n$ items, $\Omega = { x_{1} \ldots x_{n} }$, of which we wish to retain a subset $\X \subseteq \Omega$. Our requirement is that $\X$ be diverse, i.e., that it should not contain items that are too much alike, or, put differently, that it be representative of the range of items found in $\Omega$. A DPP is essentially a way of picking a random  $\X$ that has this property with high probability.

We introduce DPPs formally below, but a salient feature of classical DPPs is that the cardinal of $\X$ is a random variable. Since this is not always suitable, \cite{KuleszaTaskar:FixedSizeDPPs} have introduced a fixed-size variant (so-called $k$-DPPs), which are nothing more than DPPs conditioned on the event that $|\X| = k$. $k$-DPPs share some features with DPPs but unfortunately lose some tractability.

In this work, we show that this loss of tractability only matters for very small $n$. In large sets, $k$-DPPs and DPPs converge in a sense we make precise below, but roughly means that the probability that item $x_{i}$ ends up in set $\X$ is almost the same under a $k$-DPP and a matched DPP. Moreover, this is true for bi-inclusions (i.e., the event that $x_{i}$ and $x_{j}$ are in $\X$) or indeed for  joint inclusion probabilities of any fixed order \footnote{To be precise: $k$-DPPs and DPPs cannot be equivalent in a strong sense, since they do not have the same support (one generates a fixed size set, the other doesn't). However, for $n$ and $k$ large enough, the probability that they include a certain fixed subset converges.} . 

Practically speaking, the ability to compute inclusion probabilities is essential when $k$-DPPs are used for importance sampling. For example, in \citep{Tremblay:DPPforCoresets}, $k$-DPPs are used to estimate averages: let $L = \sum_{i=1}^n f(x_i)$. If $\X$ is sampled from a $k$-DPP, the average $L$ can be estimated from the values of $f$ in $\X$. Since not all items have equal probability of appearing in a $k$-DPP, we have to reweight by the inverse inclusion probability to form the unbiased estimate: 
\begin{equation}
  \label{eq:importance_sampling}
\hat{L}(\X) = \sum_{i = 1 }^n  \frac{f(x_i) \I(i \in \X)}{p(i \in \X)}
\end{equation}
Here we therefore need first-order inclusion probabilities. To estimate a pairwise quantity (e.g., mean distance), we would need second-order inclusion probabilities, and so on.

Our results lead to stable and accurate approximations for inclusion probabilities, as described in section \ref{sec:alg-results}, and stable algorithms for sampling $k$-DPPs with relatively large $k$. They also clarify the links between $k$-DPPs and DPPs, and when the one should look like the other. 

The article is structured as follows: in section \ref{sec:background}, we introduce notation and recall results on DPPs and $k$-DPPs. Section \ref{sec:asymptotic-eq} contains our main theoretical results. The practical algorithms that follow are described in section \ref{sec:alg-results}. Section \ref{sec:empirical_results} contains simulation results. 

To prove our main result we use saddlepoint approximations and a perturbation argument, but readers who wish to skip the technical details will find an intuitive argument in section \ref{sec:some-intuition}, where we explain that DPPs are just exponentially relaxed $k$-DPPs. Essentially, the strict constraint $|X| = k$ that appears in $k$-DPPs is relaxed to a soft constraint in DPPs, and the difference between the soft and the hard constraint becomes irrelevant in large $n$. 

\section{Background}
\label{sec:background}
In this section we introduce notation and some basic results. 

\subsection{Notation}
\label{sec:notation}

We deal with finite ground sets, so without loss of generality we may take $\Omega =  \{1, \ldots, n\}$. Fixed subsets of $\Omega$ are then equivalent to multi-indices and noted $\ba$, with cardinality noted $|\ba|$. Random subsets are noted $\X$ or $\Y$.  Expectation is noted $E( . )$, and $\I$ is the indicator function, so that e.g., $E\left(\I(i \in \X) \right)=p(i \in \X)$.
There are two equivalent viewpoints when dealing with finite random subsets: one is to look at $\X$, a subset, as the random variable. Another is to consider binary strings of length $n$, which indicate whether item $i$ is included in $\X$. We note such strings $\bz$, and depending on context one or the other viewpoint is more convenient.
Matrices are in bold capitals, e.g., $\bL$. The identity matrix is noted $\bI$. Individual entries in a matrix are noted using capitals: $L_{ij}$ is entry $(i,j)$ in matrix $\bL$. Sub-matrices are in bold, with indices, for example $\bL_{\ba,\bb}$ is the sub-matrix of $\bL$ with rows indexed by $\ba$ and columns indexed by $\bb$. So-called ``Matlab'' notation is used occasionally, so that the submatrix formed by selecting all rows in $\ba$ is noted $\bL_{\ba,:}$, and $\bL_{:,1:k}$ is the submatrix containing the first $k$ columns. For simplicity, a single index is used if it is repeated: $\bL_{\ba}=\bL_{\ba,\ba}$. Sub-matrices and sub-vectors formed by excluding elements are noted with a minus sign, e.g.,  the index $ \ba_{-j}$ includes all elements in $\ba$ except index $j$.

\subsection{Some lemmas}
\label{sec:lemmas}

We will need two well-known lemmas in the course of this work. The first one (Cauchy-Binet) is central to the theory of DPPs, the second is an easy lemma on inclusion probabilities.

The Cauchy-Binet lemma expresses the determinant of a matrix product as a sum of products of determinants:
\begin{lemma}[Cauchy-Binet]
\label{lm:cauchy-binet}
  Let $\mathbf{M} = \mathbf{A} \mathbf{B}$, with $\mathbf{A}$ a $n \times m$ matrix, $\mathbf{B}$ a $m \times n$ matrix. We assume $m \geq n$. Then:
  \begin{equation}
    \label{eq:CauchyBinet}
    \det \mathbf{M} = \sum_{\ba, |\ba| = n} \det \mathbf{A}_{:,\ba} \det \mathbf{B_{\ba,:}}
  \end{equation}
where $\ba$ is a multi-index of length $n$. The sum is over all multi-indices $\ba$, of which there are $m \choose n$.
\end{lemma}
  
The second lemma is an easy lemma on sums of inclusion probabilities. An inclusion probability is the probability that a certain item (or items) appear in a random set.

\begin{lemma}[Sums of inclusion probabilities]
\label{lm:sums-inclusion-prob}

  Let $\Omega$ designate a base set of items, and $\X$ a random subset of $\Omega$. Let $\ba$ designate a fixed subset of items of cardinality $m$. $p(\ba \subseteq \X)$ is called an inclusion probability. 
  We have that: $\sum_{\ba, |\ba| = m} p(\ba \subseteq \X) = E \left(  {|\X| \choose m } \right) $, where the expectation is over the random set $\X$. In particular:
  \begin{enumerate}
  \item if $\X$ is a set of fixed size $k$, the sum equals $k \choose m$.
  \item if $m = 1$, the sum equals $E(|\X|)$
  \end{enumerate}
\end{lemma}
\begin{proof}
  \begin{align*}
    \sum_{\ba, |\ba| = m} p(\ba \subseteq \X) &= \sum_{\ba, |\ba| = m} E \left( \I (\ba \subseteq \X) \right) \\
                                        &=  E  \left( \sum_{\ba, |\ba| = m} \I (\ba \subseteq \X) \right) \\ 
&= E \left(  {|\X| \choose m } \right)
  \end{align*}

\end{proof}

\begin{rem}
  For large sets, $E \left(  {|\X| \choose m } \right) = \frac{1}{m!} E \left( |\X| (|\X| - 1) \ldots (|\X| - m + 1) \right) = \frac{1}{m!} E(O(|\X|^{m}))$
  As a consequence, the sum of order-$m$ inclusion probabilities for a set of fixed size $k$ is $O(\frac{k^{m}}{m!})$. We use this fact to properly normalise the total variation distance, see section \ref{sec:asymp-eq}.
\end{rem}

% Next, we state a version of Scheffe's lemma:

% \begin{lemma}[Scheffe's lemma]
% \label{lm:scheffe}
%   Let $p$ designate a fixed density with respect to some measure $\mu$ (e.g., the counting measure). Let ${ q_{n} }$ designate a sequence of measures converging pointwise to $p$ everywhere (with respect to $\mu$). Then $q_{n}$ converges in total variation to $p$ if and only if $\int q_{n} d \mu = \int p d \mu$.
% \end{lemma}
  
% Scheffe's lemma, along with lemma \ref{lm:sums-inclusion-prob}, will be used to prove convergence in total variation. 
\subsection{Elementary symmetric polynomials}
\label{sec:ESPs}

The Elementary Symmetric Polynomials (ESPs) of a matrix play an important role in the theory of $k$-DPPs, and one of our core problems will be to find asymptotic formulas for them. Let $\bL$ denote a positive definite matrix and $\lambda_{1} \ldots \lambda_{n}$ its eigenvalues.
The $k$-th ESP is a sum of all the products of $k$ eigenvalues:
\begin{equation}
  \label{eq:ESP}
  e_k(\bl) = \sum_{ \ba, |\ba| = k } \prod_{j \in \ba } \lambda_{j}
\end{equation}
For example, $e_{2}(\bl) = \sum_{i <j} \lambda_{i}\lambda_{j}$. Interesting special cases include $e_{1}(\bl) = \sum \lambda_{i} = \text{Tr}(\bL)$, and $e_{n}(\bl) = \prod \lambda_{i} = \det \bL$.
There is a rich theory on ESPs, going back at least to Newton, with interesting modern developments \citep{MarietSra:ESPsOptimalExpDesign, JozsaMitchinson:SymmPolyInfTheory}. As we explain below, they occur in $k$-DPPs as normalisation constants, and ratios of ESPs appear in inclusion probabilities. 

\subsection{DPPs}
\label{sec:DPPs}

DPPs are defined such as to produce random subsets that are not overly redundant, where the notion of redundancy is defined with respect to a (positive definite) similarity function.

We have a collection $\Omega$ of items ordered from 1 to $n$. We associate to each pair of items a similarity score $L_{ij}$, such that the matrix $\bL$ with entries $L_{ij}$ is positive definite. The matrix $\bL \in \R^{n \times n}$ is called the L-ensemble of the DPP \footnote{We find it more natural to define DPPs via the L-ensemble, since the more common definition via the marginal kernel does not carry over to fixed-size DPPs.}.
\begin{definition}
A Determinantal Point Process is a random subset  $\X$ of $1 \ldots n$ with probability mass function given by:
\begin{equation}
  \label{eq:dpp_pmf}
  p(\X) = \frac{ \det(\bL_{\X})  }{ \det( \mathbf{I} + \bL ) } 
\end{equation}
\end{definition}
The preference for diverse subsets built into DPPs comes from the fact that if a subset $\X$ includes items that are too similar, the matrix $\bL_{\X}$ will have nearly colinear columns, and its determinant will be close to 0. 

An interesting aspect of DPPs is how tractable the marginals are. The inclusion probabilities, i.e.,  the probability that item $i$ is in $\X$, are given by the so-called ``marginal kernel'' matrix $\bK\in \R^{n \times n}$, where
\begin{equation}
  \label{eq:dpp_kernel}
  \bK = (\mathbf{I} + \bL)^{-1}\bL
\end{equation}
Specifically, for a DPP,  $p(i \in \X) = K_{ii} $.
More generally, inclusion probabilities are given by principal minors of the marginal kernel, e.g., if $\ba$ is a subset of $\Omega$:
\begin{equation}
  \label{eq:trace-kernel}
  p(\ba \subseteq  \X) = \det (\bK_{\ba})
 \end{equation}

A DPP can generate random subsets of any size from 1 to $n$. The expected cardinality of $\X$ can also be read out from the marginal kernel, specifically:
\begin{equation}
  \label{eq:trace-kernel}
  E(|\X|) = \text{Tr} (\bK) = \sum \frac{\lambda_{i}}{1+\lambda_{i}}
  \end{equation}
where the $\lambda_{i}$'s designate the eigenvalues of the L-ensemble $\bL$. 
  
\subsection{$k$-DPPs}
\label{sec:kDPPs}

\begin{definition}
A $k$-DPP is a DPP conditioned on the size of the sampled set $|\X| = k$. In other words, the probability mass function stays the same but now the sample space is the set of subsets of $1\ldots n$ of size $k$, and
\begin{equation}
    \label{eq:kdpp_pmf}
  p (\X\big\vert | \X| = k) \propto  \begin{cases}
    \det(\bL_{\X}) & \mathrm{if \ } \vert \X \vert = k \\
    0 &  \mathrm{otherwise}
    \end{cases}
  \end{equation}
\end{definition}
\begin{rem}
Contrary to DPPs, $k$-DPPs are insensitive to the overall scaling of the L-ensemble. Since $$ \det(\beta\bL_{\X}) = \beta^k \det(\bL_{\X}), $$ the probability density (\ref{eq:kdpp_pmf}) is invariant to any rescaling by a factor $\beta > 0$. 
\end{rem}
An important property of $k$-DPPs, one that unlocks many analytical simplifications, is that $k$-DPPs are a mixture distribution. The mixture involves a diagonal $k$-DPP and a projection $k$-DPP, two objects that are simpler than a generic $k$-DPP.

The mixture property is a consequence of the Cauchy-Binet formula (lemma \ref{lm:cauchy-binet}). Let $\bL = \mathbf{UDU}^\top$ denote the spectral decomposition of $\bL$, with $\mathbf{D} = \mathrm{diag}(\lambda_{1} \ldots \lambda_{n})$ and $\mathbf{U}$ the matrix of eigenvectors. Then
\begin{equation}
  \label{eq:cauchy-binet-mixture}
   p\left(\X \big\vert | \X| = k\right) = \frac{1}{Z} \det(\bL_\X)= \frac{1}{Z} \sum_{\Y,|\Y|=k} \det(\mathbf{U}_{\X,\Y} \mathbf{U}_{\X,\Y}^\top)\det(\mathbf{D}_{\Y} )
\end{equation}
where $Z$ is an integration constant (to be defined later), $\Y$ is a subset of \emph{columns} of $\mathbf{U}$, and the sum is over all such subsets of size $k$.
Equation \eqref{eq:cauchy-binet-mixture} shows that  the probability mass function has the form of a mixture distribution, where we first choose a set of \emph{eigenvalues} (with indices $\Y$) from a $k$-DPP with  \emph{diagonal} L-ensemble $\mathbf{D}$ and then choose a set of items $\X$ from a $k$-DPP with L-ensemble $\mathbf{U}_{:,\Y} \mathbf{U}_{:,\Y}^\top$. The latter is a specific kind of DPP, called a ``projection DPP'' . 

The same mixture interpretation holds for DPPs as well. In the case of DPPs, the rule for sampling the set $\Y$ of eigenvalues is simpler. Each eigenvalue is sampled independently and included with probability $\frac{\lambda_{i}}{1+\lambda_{i}}$. Once we have the eigenvalues, we proceed in exactly the same way as above: form a projection kernel, and sample the corresponding projection DPP.

\subsubsection{Projection DPPs}
\label{sec:projection_dpp}

\begin{definition}
A projection DPP is a $k$-DPP whose L-ensemble has the following form:
\begin{equation}
  \label{eq:proj_dpp}
  \mathbf{L} = \mathbf{VV}^\top
\end{equation}
where $\mathbf{V}_{n \times k}$ has orthonormal columns (i.e.,  $\mathbf{V}^\top\mathbf{V} = \mathbf{I}$).
\end{definition}

Projection DPPs have a set of properties that make them especially tractable. The most salient is that the marginal kernel equals the L-ensemble, e.g., the inclusion probability of item $i$ equals $L_{ii}$, as shown in the following lemma.

\begin{lemma}
  \label{lm:marginal-kernel-dpp}
  In a projection DPP with L-ensemble  $ \mathbf{L} $, $p \left(\ba \subseteq \X \right) = \det(\bL_\X)$. 
\end{lemma}
\begin{proof}
  See appendix. 
\end{proof}

This result is proved rigorously in the appendix, but straightforward if one looks at projection DPPs as DPPs taken to a certain limit. Consider a DPP with the following L-matrix, indexed by parameter $ \gamma > 0 $:
\begin{equation}
  \label{eq:dpp-limit-pddp}
  \bL(\gamma) = \bR \mathbf{D(\gamma)} \bR^\top 
\end{equation}
where $D(\gamma)$ is a diagonal matrix with entries on the diagonal equal to $\gamma$ repeated $k$ times, followed by $\gamma^{-1}$, repeated $n-k$ times, and $\bR$ is a $n \times n$ orthonormal matrix.
Let $\gamma \rightarrow \infty$. Following the mixture interpretation of DPPs, we see that the probability of picking one of the first $k$ eigenvalues equals $\gamma/(1+\gamma)$, which tends to 1, while the probability of picking one of the latter $n-k$ tends to 0. This means that with increasing $\gamma$ we end up always picking the same $k$ eigenvalues, and hence always sampling the same $k$-DPP, one with kernel $\bR_{:,1:k}\bR_{:,1:k}^\top$. The marginal probabilities are given by the corresponding marginal kernel:
$ \bR \mathbf{D_{m}(\gamma) } \bR^\top $
where $D_{m} (\gamma)$ has first $k$ entries equal to $ \frac{\gamma}{1+\gamma}$, and the next $n-k$ equal to $\frac{1}{\gamma+1}$. In the large-$\gamma$ limit, the marginal kernel thus equals $\bR_{:,1:k}\bR_{:,1:k}^\top$ as claimed. The limit is however improper, as some entries in the L-matrix tend to infinity.

To sum up: if the L-ensemble is a projection matrix of rank $k$, then a $k$-DPP is also a DPP. We can even extend this further to \emph{all} L-ensembles of rank $k$.

\begin{result}
  \label{result:max-rank-dpp}
  Let $\bL$ have rank $k$, with eigendecomposition $\bL = \bU \mathbf{D} \bU^{\top}$. Without loss of generality, we assume that $\bU$ is of size $n \times k$ and $\mathbf{D}$ a diagonal matrix of size $k \times k$ with non-null diagonal elements. Then a $k$-DPP with L-ensemble $\bL$ is also a projection DPP, with marginal kernel equal to $\bU \bU^{\top}$. 
\end{result}
\begin{proof}
  $\bL$ has rank $k$, so in the eigendecomposition $\bU$ is $n \times k$, and $\mathbf{D}$ is a diagonal matrix of size $k \times k$. If $\X$ is a subset of size $k$, we have
  \[ \det \bL_{\X} = \det \bU_{\X,:} \mathbf{D} \bU_{:,\X}^{\top} \]
  and since the matrices involved are square, we have:
  \[ \det \bL_{\X} = \det D \left( \det \bU_{\X,:}\bU_{:,\X}^{\top} \right) \]
  Then $p(\X) \propto \left( \det \bU_{\X,:}\bU_{:,\X}^{\top} \right) $, which is the probability mass function of a projection DPP and the result follows.
\end{proof}

This result hints at a close kinship between $k$-DPPs and DPPs, and convergence results bear this out. 

\subsubsection{Inclusion probabilities in $k$-DPPs}
\label{sec:inclusion_kDPPs}

Since a $k$-DPP is a mixture of projection-DPPs (eq. \ref{eq:cauchy-binet-mixture}), the first order inclusion probability for item $i$ can be expressed as 
\begin{align}
  p\left(i \in \X \vert | \X| = k\right) &= E_\Y( (\mathbf{U}_\Y \mathbf{U}_\Y^\top )_{ii}) \\
                                         &= E_\Y( \sum_{j=1}^n U_{ij}^2  \mathrm{I}(j \in \Y)) \\
                                         &= \sum_{j=1}^n U_{ij}^2  P(j \in \Y) \\
                                         &= (\mathbf{U}\mathrm{diag}(\bm{\pi})\mathbf{U}^\top)_{ii}
\end{align}
where $\pi_j = p(j \in \Y)$, the probability that the $j$-th eigenvector is included in set $\Y$. Formulas for higher-orders (joint inclusion probabilities) are in section \ref{sec:reduction_diag_proof}. 

Computing the inclusion probabilities for a $k$-DPP thus boils down to computing inclusion probabilities in a \emph{diagonal} $k$-DPP, and combining them with the eigenvectors of $\bL$. 

\subsection{Diagonal DPPs and $k$-DPPs}
\label{diag-DPPs}

In the special case of diagonal DPPs and $k$-DPPs, the L-ensemble is a diagonal matrix. A diagonal DPP turns out to be nothing more than a Bernoulli process. If conditioned to be of fixed size $k$, a diagonal $k$-DPP is obtained.

So far we have kept with the usual viewpoint on DPPs, which sees them as random sets. Alternatively, a sample from a discrete DPP can be viewed as a binary string $\bz$ of size $n$, where $z_i = 1$ indicates inclusion of the $i$-th item, and $\sum_{i=1}^n z_i = k$. In this section we prefer the latter viewpoint, because it lightens notation. 

In this notation the inclusion probability of item $i$ equals the marginal probability of $z_{i}$, $p(z_{i}=1)$, and similarly for joint probabilities $p(z_{i}=1,z_{j}=1)$, etc. $p(\bz) = p(z_{1}\ldots z_{n})$ is the likelihood of the draw.

\subsubsection{Diagonal DPPs}

Consider a DPP with diagonal L-ensemble
\begin{equation*}
  \bL = \diag ( \lambda_1, \ldots, \lambda_n) 
\end{equation*}
Following eq. (\ref{eq:dpp_kernel}), $\bK$ is diagonal too, with entries $K_{ii} = \pi_{i} =  \frac{\lambda_{i}}{1+\lambda_{i}}$. The fact that the marginal kernel is diagonal implies that
$p(z_{i}=1,z_{j}=1) = \det(\bK_{\left\{ i,j \right\}}) = \pi_{i}\pi_{j} = p(z_{i}=1)p(z_{j}=1)$, with similar results for higher-order probabilities. We conclude that (viewed as a binary string) a diagonal DPP is a product of independent Bernoulli variables, where each $z_{i}$ is drawn with probability $\pi_{i}$. 

\subsubsection{Diagonal $k$-DPPs}

Viewed as distributions over binary strings, diagonal DPPs are a product measure, meaning that each $z_{i}$ is sampled independently. Diagonal $k$-DPPs are not, due to the constraint that $\sum z_{i} = k$. 
The density of a diagonal $k$-DPP is given by: 
\begin{equation}
  \label{eq:prob-density-dkDPP}
  p(\bz) =   \frac{\prod_{j =1 }^{n} \lambda_{j} ^{z_{j}}}{Z}   \I (\sum z_{i} = k) 
\end{equation}
The integration constant $Z$ is given by the $k$'th elementary symmetric polynomial (ESP)
\begin{equation}
  \label{eq:ESP}
  Z = e_k(\bl) = \sum_{ \ba } \prod_{j \in \ba } \lambda_{j}
\end{equation}
where $\ba$ is a multi-index of size $k$.
At this stage, it may be hard to see what sort of probability distribution eq. (\ref{eq:prob-density-dkDPP}) defines. Indeed, it is not obvious how to sample from such a distribution, and the algorithm given in \cite{KuleszaTaskar:DPPsforML} is not trivial. We return to the issue in section \ref{sec:sampling-diag}. 

Inclusion probabilities can be computed through direct summation. 
\begin{align}
  \label{eq:inclusion_prob_diag}
  p(z_{i}=1) &= \sum_{\bz_{-i}} p(z_{i}=1,\bz_{-i})
               = \frac{\lambda_i  \sum_{ { |\ba| = k-1 , \ba \cap \left\{ i \right\} = \emptyset} } \prod_{j \in \alpha } \lambda_{j} }{e_k(\bl)} \\
  &= \frac{\lambda_ie_{k-1}(\bl_{-i})}{e_k(\bl)} 
  \end{align}
Computing such quantities in practice is again not completely trivial, although \citep{KuleszaTaskar:DPPsforML} gives an algorithm. We include a fairly accurate approximation below, and due to numerical instabilities in the exact algorithm, we advocate using the approximation in most cases (Section \ref{sec:empirical_results}). 

\section{Asymptotic equivalence of $k$-DPPs and DPPs}
\label{sec:asymptotic-eq}

Before stating our main results formally, we give an intuitive argument as to why $k$-DPPs and DPPs may resemble one another. 

\subsection{Some intuition}
\label{sec:some-intuition}

Readers familiar with statistical physics will know of a class of results known as ``equivalence of ensembles'' \citep{Touchette:EquivAndNonequiv}. These results justify formally a mathematical subterfuge, whereby a probability distribution that incorporates a hard constraint (the ``micro-canonical ensemble'') can be replaced with a more tractable variant (the ``canonical ensemble''), where the hard constraint is turned into a soft constraint. Our result is a variant of this particular scenario. 

We rewrite the likelihood of a $k$-DPP as the likelihood of a DPP times a hard constraint:
\begin{equation*}
  p(\X) \propto \left( \det \bL_{\X} \right) \I ( |\X| = k)
\end{equation*}
Deploy now the usual trick of turning the hard constraint into a soft constraint via an exponential, defining a new distribution:
\begin{equation}
  \label{eq:$k$-DPP-softened}
  q(\X) \propto \left( \det \bL_{\X} \right) \exp (  \nu |\X|)
\end{equation}
where $\nu$ should be set so that $|\X|=k$ on average over $q$, i.e.,  $E_{q}(|\X|) = k$. Before we find such a value, it helps to recognise that $q$  actually has the form of a DPP: since $\det(\beta \bL_{\X}) = \beta^{|\X]} \det \bL_{X}$, we have
\begin{equation}
  \label{eq:k-DPP-softened}
  q(\X) \propto  \det \left( \exp(\nu) \bL_{\X} \right) 
\end{equation}
and we identify $q$ as a DPP with L-ensemble $\exp(\nu) \bL_{\X}$.
Using eq. (\ref{eq:trace-kernel}), we find that:
\begin{equation}
  \label{eq:nu_implicit}
E_q(|\X|) = \exp(\nu) \Tr \left( (\exp(\nu) \bL + \bI)^{-1} \bL \right) 
\end{equation}
The appropriate value for $\nu$ is determined by the implicit equation that $E_q(|\X|) = k$. In terms of the eigenvalues, this reads: 
\begin{equation}
  \label{eq:nu_implicit}
  \sum_i \frac{\lambda_i e^\nu}{1+\lambda_ie^\nu} = k
\end{equation}

To sum up, this development suggests that a $k$-DPP with ensemble $\bL$ can be approximated by a (tilted) DPP with L-ensemble $\exp(\nu) \bL$, with $\nu$ set so that the matched DPP has $k$ elements on average. The next section gives a rigorous statement for this approximation. 

\subsection{Main result}
\label{sec:asymp-eq}

Under certain conditions, DPPs and $k$-DPPs are equivalent in a regime where we pick a fixed ratio of items from a growing set, i.e.,  $\frac{k}{n}=r > 0$, fixed as $n \rightarrow \infty $. By equivalence, we mean that they have the same marginals (inclusion probabilities of order 1 and above). The conditions for equivalence boil down to the \emph{number of degrees of freedom of $\bL$ being high enough}, and we make that condition more precise below. In practice the approximations we derive give excellent results in most settings we have tried, except with very small values of $n$ (less than 10, say). 

We require assumptions on the L-ensembles: let $\bL_{1} \ldots \bL_{n}$ denote a sequence of positive definite matrices of increasing size $n \times n$. The assumption is that  $\Tr \left( (\bL_{n} + \bI)^{-2}\bL_{n}  \right) $ diverges. The question of which sequences of matrices verify this condition is left to section \ref{sec:which_matrices}.

We associate with each $\bL_{n}$ a $k$-DPP $\X_n$, where $k = \lfloor rn \rfloor$, a fixed fraction of the number of items. Similarly, we have a second sequence of \emph{matched} DPPs $\tilde{\X}_{n}$ with L-ensemble $\exp(\nu_{n})\bL_{n}$, where $\nu_{n}$ verifies eq. (\ref{eq:nu_implicit}). 
Let $\ba$ denote a multi-index of fixed finite size $m < k$, and $\pi_{n}(\ba)$ the probability that $\ba \subseteq \X_{n}$, and $\tilde{\pi}_{n}(\ba)$ the corresponding probability for $\tilde{\X_{n}}$. We may interpret $\pi$ and $\tilde{\pi}$ as two measures over $\ba$, and an appropriate means of comparing these quantities is via total variation. Because $\pi$ and $\tilde{\pi}$ have total mass that grows with $k$ (see lemma \ref{lm:sums-inclusion-prob}), we normalise the total variation distance with the appropriate factor. 

\begin{definition}
  Let $\pi$, $\tilde{\pi}$ designate two inclusion measures of order $m \geq 1$, corresponding to inclusion probabilities in point processes with $n$ elements. We define their total variation distance as: 
  \begin{equation}
    \label{eq:total-variation-def}
    D_{m}(\pi,\tilde{\pi}) = {k \choose m}^{-1} \sum_{\ba, |\ba| = m} \left| \pi( \ba) - \tilde{\pi} (\ba) \right| 
  \end{equation}
\end{definition}

We have the following result:

\begin{thm}
  \label{thm:main-result}
  Under the assumptions above, joint inclusion probabilities under a $k$-DPP and its matched DPP converge: 
  \begin{equation}
    \label{eq:lim-total-var}
 %   \lim_{n \rightarrow \infty} 
    D_m(\pi_{n},\tilde{\pi}_{n}) = O(n^{-1}) \mbox{ as } n \rightarrow \infty
  \end{equation}
\end{thm}

\begin{rem}
Note that in our proof we have $k = O(n)$, which is needed because of a Central Limit argument implicit in the saddlepoint expansion.
  
\end{rem}
\begin{rem}
A quantity of interest in many calculations are sample averages of the form $A(\X) = \frac{1}{m} \sum_{i \in \X} f_{i} $. Then $E_{\X}(A) =   \frac{1}{m} \sum_{j \in \Omega} \pi(j) f_{j}$. An easy corollary is that $\left| E_{\X}(A) - E_{\tilde{\X}}(A) \right| \rightarrow 0$, from well-known properties of the total variation distance \citep{DasGupta:AsymptoticTheoryStatsProb}.  
\end{rem}

The overall proof path for theorem \ref{thm:main-result} is as follows:
\begin{enumerate}
\item We reduce the equivalence of $k$-DPPs and DPPs to the equivalence of \emph{diagonal} $k$-DPPs and DPPs (section \ref{sec:reduction_diag})
\item Elementary symmetric polynomials (and ratios thereof) hold the key to the next step, and we show how they can be approximated using a saddlepoint approximation (section \ref{sec:saddlepoint-ESPs})
\item We insert the asymptotic series for ESPs into the formula for inclusion probabilities, and derive the $O(1)$ and $O(n^{-1})$ terms. The $O(1)$ term corresponds to inclusion probabilities in the matched DPP, from which Theorem \ref{thm:main-result} follows (section \ref{sec:inclusion-prob-lemma}).
\end{enumerate}

\subsubsection{Reduction to diagonal DPPs}
\label{sec:reduction_diag}

Recall (section \ref{sec:kDPPs}) that DPPs and $k$-DPPs are both mixture distributions, where we first draw a set of eigenvectors of $\bL$, and then draw from a projection DPP formed from these eigenvectors. That second step is the same in DPPs and $k$-DPPs, only the first step differs. In DPPs, we draw from a diagonal DPP, while in $k$-DPPs we draw from a diagonal $k$-DPP. Heuristically, because it is only the first step that differs, we can focus on our asymptotic study on the first step. 

Formally if we can establish that the inclusion probabilities in \emph{diagonal} $k$-DPPs and DPPs converge (at any finite order), then the inclusion probabilities in \emph{general} $k$-DPPs and DPPs converge as well (up to the same order). We note $\Y$ and $\tilde{\Y}$ the diagonal DPPs associated with $\X$ and $\tilde{\X}$. The order-$m$ inclusion measures for $\X$ and $\tilde{\X}$ are noted $\pi_{m}$ and $\tilde{\pi}_{m}$, while the corresponding measures for $\Y$ and $\tilde{\Y}$ are noted $\rho_{m}$ and $\tilde{\rho}_{m}$ (the latter correspond to the probability that certain eigenvectors are included, as per the mixture interpretation of DPPs introduced in section \ref{sec:projection_dpp}). 

The following lemma states the result:
\begin{lemma}
  \label{lm:diag-reduction}
  $D_{m}(\pi_{m},\tilde{\pi}_{m}) \leq  D_{m}(\rho_{m},\tilde{\rho}_{m})$
\end{lemma}
Lemma \ref{lm:diag-reduction} implies that if diagonal $k$-DPPs converge to matched diagonal DPPs, so do general k-DPPs. The proof is deferred to the appendix (section \ref{sec:reduction_diag_proof}). Armed with this lemma, we now focus only on the diagonal case. 

Our goal is now to compute inclusion probabilities in diagonal $k$-DPPs. Recall that $\ba$ denotes a subset of $(1, \ldots ,n)$ of fixed size $m$. We wish to compute $p(\ba \in \Y)$, or equivalently, the probability that $p( \prod_{j \in \ba} z_{j}=1 )$.
This marginal probability can be computed via direct summation:
\begin{equation}
  p( \prod_{j \in \ba} z_{j}=1 ) = \frac{\left(\prod_{j \in \ba} \lambda_{i}\right) \sum_{\bb,|\bb| = k-|\ba|,\bb  \cap \ba = \emptyset} \prod_{j \in \bb} \lambda_j }{e_{k}(\bl)}
  = \left(\prod_{i \in \ba} \lambda_{i}\right) \frac{e_{k-m}(\bl_{-\ba})}{e_{k}(\bl)}
    \label{eq:inclusion-prob-diag}  
\end{equation}

Thus, inclusion probabilities in diagonal DPPs can be expressed using ratios of ESPs. This leads us to our next section, where we derive an asymptotic approximation for ESPs. We will then insert the asymptotic approximation into eq. \eqref{eq:inclusion-prob-diag}, to get an asymptotic series for inclusion probabilities. 

\subsubsection{Saddlepoint approximation for ESPs}
\label{sec:saddlepoint-ESPs}

ESPs are unwieldy combinatorial objects, but fortunately they lend themselves well to asymptotic approximation. This section is crucial for the rest and so we keep the details in the main text. 

ESPs have an elegant probabilistic interpretation (already noted in passing in \citep{Chen:WeightedFinitePopSampling}). An equivalent definition for ESPs views them as the coefficients in a power series: 
\begin{equation}
  \label{eq:gen_func_esp}
  e_{k}(\bl) = [x^{k}] \prod_{i=1}^n (1+ \lambda_i x)
\end{equation}
We borrow the notation $[x^{k}] f(x)$ from combinatorics to denote the coefficient of $x^{k}$ in the series $f$.
To uncover the probabilistic interpretation of ESPs, we transform the series into a \emph{probability} generating function.
\begin{align}
  \label{eq:gen_func_esp}
  e_{k}(\bl) &= [x^k] \prod_{i=1}^n (1+\lambda_i)\frac{(1+ \lambda_i x)}{1+\lambda_i} \\
  &= \prod_{i=1}^n (1+\lambda_i)  [x^k] \prod (1-p_i + x p_i )
\end{align}
where $p_i = \frac{ \lambda_i }{ 1 + \lambda_i } \in (0,1)$ is now to be interpreted as the parameter of a Bernoulli variable, $B_i$. Let $S_n = \sum_{i=1}^n B_i$ designate the sum of all such independent $B_i$'s.
Then:
$$ p(S_n = k) = [x^k] \prod_{i=1}^n (1-p_i + x p_i ) = \frac{e_k(\bl)}{\prod_{i=1}^n (1+\lambda_i)} $$
Since $S_n$ is the sum of $n$ independent random variables, it invites a central limit approximation to the $p(S_n=k)$. 
First, note that:
\begin{equation}
  \label{eq:mu}
  \mu = E(S_n) = \sum \frac{  \lambda_i }{1+\lambda_i} 
\end{equation}
which tells us that $e_k$, taken as a function of $k$, is likely to peak near $\mu$. 
The second moment,
\begin{equation}
  \label{eq:sigma2}
 \sigma^2 = Var(S_n) =  \sum \frac{ \lambda_i }{(1+\lambda_i)^2} 
\end{equation}
gives a measure of scale for the peak of  $e_k$ around $\mu$. Since $\frac{ \lambda_i }{(1+\lambda_i)^2} \leq \frac{\lambda_i}{1+\lambda_i}$, we have:
\begin{equation}
  \label{eq:var_lt_mu}
  \sigma^2 \leq \mu
\end{equation}

In studying the convergence of $k$-DPPs and DPPs, it is $\sigma^{2}$, rather than $\mu$ that captures the appropriate notion of  ``degrees of freedom''. In our case the Lyapunov Central Limit Theorem \citep{Billingsley:ProbAndMeasure} requires that $\sigma^{2}$ diverge asymptotically, and the condition we assumed on the sequence of L-ensembles guarantees exactly that (see section \ref{sec:which_matrices} for a discussion).

A much better approximation than the Gaussian CLT  is the saddlepoint approximation of \citep{Daniels:SaddlepointApprox}. Unlike the CLT, it is accurate in the tails and has $O(n^{-1})$ relative error. It reads:
\begin{equation}
  \label{eq:saddlepoint-approx}
  p(S_{n}=k) = \frac{1}{\sqrt{2\pi \psi''(\nus})}\exp \left(\psi(\nus) - k\nus \right) \left(1+O(n^{{-1}})\right)
\end{equation}
where $\psi(\nu) = \log E \left( \exp (\nu S_{n}) \right)$ is the cumulant-generating function of $S_{n}$, and $\nus$ is the solution of the saddlepoint equation:
\begin{equation}
  \label{eq:saddlepoint-eq-generic}
\nus = \argmin_{\nu}{\psi(\nu) - k\nu}
\end{equation}

In our case, we have:
\begin{align*}
   \psi(\nu) &= \log E \left( \exp (\nu S_{n}) \right) \\
   &= \sum_{i=1}^{n} \log E \left( \exp (\nu B_{i}) \right) \\
   &= \sum \log \left(\frac{1}{1+\lambda_{i}} + \frac{\lambda_{i}}{1+\lambda_{i}}e^\nu \right) \\
    &= \sum \log \left(1+\lambda_{i}e^\nu \right) - \sum \log \left(1+\lambda_{i} \right)   \numberthis \label{eq:cumulants}
    \end{align*}
We will need  the derivatives of $\psi$ as well:
  \begin{equation}
    \label{eq:cumulants-first-derivative}
\psi'(\nu) = \sum \frac{ \lambda_{i}e^\nu  }{1+\lambda_{i}e^\nu}
  \end{equation}
  \begin{equation}
    \label{eq:cumulants-second-derivative}
\psi''(\nu) = \sum \frac{\lambda_{i}e^\nu }{(1+\lambda_{i}e^\nu)^{2}}
  \end{equation}

  Inserting \eqref{eq:cumulants} into \eqref{eq:saddlepoint-eq-generic}, we see that:
  \begin{equation*}
     \sum \frac{ \lambda_{i}e^{\nus}  }{1+\lambda_{i}e^{\nus}} = k
  \end{equation*}
recovering \eqref{eq:nu_implicit}. 

To summarise: inserting \eqref{eq:gen_func_esp} into \eqref{eq:saddlepoint-eq-generic}, we have
\begin{lemma}
\begin{equation}
  \label{ eq:ESP-saddlepoint}
  e_{k}(\bl) = \frac{1}{\sqrt{2\pi \psi''(\nus}) }\exp \left(\sum_{i=1}^{n}( \log(1+\lambda_{i}e^{\nus})) - k\nus \right) ( 1 + O(n^{-1}))
\end{equation}
\end{lemma}

\begin{rem}
In large $n$ the exponential term dominates (a large deviation regime, see \citet{Touchette:EquivAndNonequiv}), and we have: 
\begin{equation}
  \label{ eq:ESP-saddlepoint-large-dev}
 \log  e_{k}(\bl) \approx \sum_{i=1}^{n} \log \left(1+\lambda_{i}e^{\nus} \right) - k\nus
\end{equation}

In random matrix theory it is customary to define the \emph{Shannon transform} of a matrix $\bL$ as:
$T(s) = \log \det( \bI + s\bL )$ \citep{Couillet:RandomMatrixMethods}. Eq. \eqref{ eq:ESP-saddlepoint-large-dev} says that for large matrices, the ESPs of $\bL$ are directly related to the Legendre transform of $T(e^{\nu})$. 
\end{rem}

At this stage, we have a tractable approximation to ESPs, and we are now ready to use it to find an approximation for inclusion probabilities.  

\subsubsection{Inclusion probabilities, and ratios of ESPs}
\label{sec:inclusion-prob-lemma}

To study the asymptotics of inclusion probabilities, we insert approximation \eqref{eq:saddlepoint-approx} into eq. \eqref{eq:inclusion-prob-diag}, and compute the $O(1)$ and $O(n^{-1})$ terms. The calculation is lengthy and can be found in the appendix (section \ref{sec:asymptotics-inclusion-proof}). The end result is as follows:

\begin{lemma}
  \label{lm:inclusion-probs}
In a diagonal k-DPP $\Y$ with L-ensemble $\mathrm{diag}(\lambda_{1},\ldots,\lambda_{n})$, inclusion probabilities have the asymptotic form: 
\begin{equation}
   \label{eq:incl-prob-improved}
   p_{k}( \ba \in  \Y ) = \left( \prod_{i \in \ba} \frac{\lambda_{i} \exp(\nus)}{1+\lambda_{i}\exp(\nus)}  \right)
   \left( 1 + \frac{1}{n} g(\nus)
     + O\left(\frac{1}{n^{2}} \right) \right)
 \end{equation}
 with
 \begin{equation*}
      g(\nus) = -\frac{\nu_{1}^{2}}{2}\psib''(\nus)
      - \frac{1}{ 2 \psib''(\nus)}  \left( \psib^{(3)} ( \nus) \nu_{1} - m \psib_{\ba}'' (\nus) \right)
 \end{equation*}

 The terms appearing in the correction $g(\nus)$ are defined in appendix \ref{sec:asymptotics-inclusion-proof}.
\end{lemma}

Notice that the $O(1)$ term corresponds exactly to the inclusion probability in the matched diagonal DPP, $\tilde{\Y}$. We now have all the elements we need to prove Theorem \ref{thm:main-result}.
%which is obtained directly from Lemma \ref{lm:inclusion-probs} by summing all error terms and applying Lemma \ref{lm:diag-reduction}. 
Consider a $k$-DPP with $m$-th order inclusion probability  $\pi_{m}$. Let $\tilde{\pi}_{m}$ be the $m$-th order inclusion probability of the matched DPP. Let the corresponding measure  for the generating diagonal $k$-DPP be  $\rho_{m} ( \ba)=p_{k}( \ba \in  \Y ) $, whose approximation  $\rho_{m} =\tilde{\rho}_{m} (1+ O(1/n))$ is given by eq. (\ref{eq:incl-prob-improved}). Starting with Lemma \ref{lm:diag-reduction} and using the approximation leads to 
\begin{eqnarray*}
D_{m}(\pi_{m},\tilde{\pi}_{m}) &\leq & D_{m}(\rho_{m},\tilde{\rho}_{m}) \\
%&=&  {k \choose m}^{-1} \sum_{\ba, |\ba| = m} \left| \rho_{m}( \ba) -\tilde{\rho}_{m} (\ba) \right| \\
&=& {k \choose m}^{-1} \sum_{\ba, |\ba| = m} 
\left( \prod_{i \in \ba} \frac{\lambda_{i} \exp(\nus)}{1+\lambda_{i}\exp(\nus)}  \right)
 \left(\frac{1}{n} g(\nus) + O\left(\frac{1}{n^{2}} \right) \right) \\
 &\stackrel{(a)}{=}&k^m {k \choose m}^{-1}  \left(\frac{1}{n} g(\nus) + O\left(\frac{1}{n^{2}} \right) \right) \\
 &\stackrel{(b)}{=}&O\left(\frac{1}{n} \right)
 \end{eqnarray*}
where equality $(a)$ is due to  equation (\ref{eq:nu_implicit}) which implicitly defines $\nus$, and equality $(b)$ holds because ${k \choose m}=O(k^m)$. This concludes the proof of the main result. A refinement is described in Appendix \ref{sec:corrected_approx}, where we derive a tractable correction to multivariate inclusion probabilities. 

A remark on the precise nature of the convergence result is in order. Regardless of how large $n$ is, a $k$-DPP will continue to produce sets of fixed size, while a DPP will continue to produce sets of variable size. This implies that DPPs and $k$-DPPs cannot be equivalent in the very strong sense of the respective probability mass functions agreeing on every possible set, since by definition they remain different. The result is of the same nature as \emph{equivalence of ensembles} in statistical physics: it pertains to two different distributions that agree more and more as $n$ tends to infinity, but never agree completely. Practically speaking, an interpretation is that for a given $n$, a $k$-DPP and a matched DPP will have very similar moments up to a certain order: certainly, at order $m > k$, this cannot be true, since the inclusion measure for the $k$-DPP is uniformly zero, but that is not true for the DPP. To get agreement up to higher orders, one has to increase $n$. 

Besides the main result, another consequence of lemma \ref{lm:inclusion-probs} is that in importance sampling estimators of the form given by eq. \eqref{eq:importance_sampling} can be used with approximate rather than exact probabilities. Using the $O(n^{-1})$ approximation induces order $O(n^{-1})$ bias, and similarly using the $O(n^{-2})$ correction induces order $O(n^{-2})$ bias. Our recommendation is therefore that one samples $k$-DPPs, rather than DPPs, while using the approximate inclusion probabilities in computations. 

\subsection{To which sequences of matrices does this apply?}
\label{sec:which_matrices}
We stated earlier that the result applies to any sequence of matrices whose degrees of freedom grow as a function of $n$, with the more precise statement being that $\Tr \left( (\bL_{n} + \bI)^{-2}\bL_{n}  \right)$ (see eq. (\ref{eq:sigma2})) should diverge. With the caveat that the condition is sufficient and not necessary, in what sort of scenarios can we expect it to hold?

A full discussion of the issue would require significant forays into random matrix theory and take us beyond the scope of the current work, so we only give a sufficient condition that is  relatively easily checked. As mentioned in section \ref{sec:kDPPs}, in $k$-DPPs, the L-ensemble can be multiplied by an arbitrary positive constant without changing the distribution. This means that we are free to scale each $\bL_{n}$ by an arbitrary constant independently for each $n$, a normalisation that lets us for instance set $\lambda_{max}$ to 1 for all $n$.
For $x \leq 1$,  $\frac{x}{(1+x)^{2}} \geq \frac{1}{4} x$, which implies that $\sigma^{2}(n) \geq \frac{1}{4} \Tr(\bL_{n})$, and a sufficient condition for the theorem to apply is therefore that $\Tr(\bL_{n})$ diverges. 

To pick a practical scenario, consider ``in-fill'' asymptotics. We suppose that the original set of data is made up of $n$ vectors in $\R^{d}$ sampled i.i.d. from a density $\rho(\mathbf{x})$.  The L-ensemble used is the classical squared-exponential (Gaussian) kernel. Let $\bL_{n} =\frac{\bM_{n}}{\lambda_{max}(\bM_{n})} $, where $M_{ij} = \exp \left(-\frac{1}{2\tau^{2}} || \mathbf{x}_{i} - \mathbf{x}_{j}  ||^{2} \right) $. $\Tr \bM = n$, and from the Gershgorin circle theorem we have a bound on $\lambda_{max}$ that reads $\lambda_{max} \leq \max_{i} \sum_{j} M_{ij}$. A sufficient condition for convergence is then that $\frac{n}{\max_{i} \sum_{ j} M_{ij} }$ diverges, which will not be the case for fixed $\tau$. The reason is that $\sum_{ j} M_{ij} = \sum \exp \left(-\frac{1}{2\tau^{2}} || \mathbf{x}_{i} - \mathbf{x}_{j}  ||^{2} \right) $ essentially counts the number of points in a neighbourhood of size $\tau$ around $\mathbf{x}_{i}$, and that quantity is $O(n)$. To make $\Tr(\bL_{n})$ diverge, we need to shrink $\tau$ with $n$ so that each point has $O(1)$ neighbours.   Similarly, the condition holds under so-called ``increasing-domain'' asymptotics, in which $\tau$ is fixed but we consider points in a growing window.
It is likely that one could relax these criteria, but in any case we must emphasise that (a) the approximations work really well in practice, see section \ref{sec:empirical_results}  and (b) actual simulations of $k$-DPPs require L-ensembles that have effective rank quite a bit larger than $k$, otherwise the numerical difficulties are overwhelming even though the process may be well defined. 

\subsection{Consequences for inference}
\label{sec:csq-for-inference}

DPPs are not only used for sampling, but also as statistical models for certain types of data that exhibit repulsion. Now in this case as well the modeller has to make a choice, and use either $k$-DPPs or DPPs. The former seems to imply that the number of observations (which is the role played here by $k$) is known in advance, while the latter does not. Interestingly, the results above imply that the choice of fixed size or varying size is of no consequence, at least if maximum likelihood is used for inference, though we suspect that Bayesian inference would be the same in that regard. 
To be precise, what we have in mind here is a case in which we observe a set $\X$ of size $k$, assumed to have been drawn from a $k$-DPP of matrix $\bL(\bt)$, where $\bt$ is a vector of parameters. For instance, $\bt$ may control the amount of repulsion in the point process.
The log-likelihood of a $k$-DPP is given by:
\begin{equation}
  \label{eq:loglik-kdpp}
  \mathcal{C}_{kDPP}(\bt) =  \log \det(\bL(\bt)_{\X}) - \log e_{k}\left(\bl_{\bt}\right) 
\end{equation}
The corresponding Maximum Likelihood estimator of $\bt$ is noted:
\begin{equation}
  \label{eq:mle-kdpp}
\hat{\bt}_{kDPP} = \argmin \mathcal{C}_{kDPP}(\bt)
\end{equation}

Similarly, a DPP model would assume $\X$ to be drawn from a DPP with L-ensemble $e^\nu \bL(\theta)$, where $e^\nu$ controls the expected cardinality of the set.
The log-likelihood reads in this case:
\begin{equation}
  \label{eq:loglik-dpp}
\mathcal{C}_{DPP}(\bt,\nu) = \nu k  + \log \det(\bL(\bt)_{\X}) - \log \det \left( \bI + e^\nu \bL(\bt) \right)
\end{equation}

Since $\nu$ is effectively a nuisance parameter, we may use a profile likelihood:
\begin{equation}
  \label{eq:loglik-profile}
\mathcal{C}^{\star}_{DPP}(\bt) = \max_{\nu} \mathcal{C}_{DPP}(\bt,\nu)
\end{equation}
The ML estimator of $\bt$ in this case solves:
\begin{align*}
  \label{eq:ML-estimator-dpp}
  \hat{\bt}_{DPP} &= \argmax_{\bt} \mathcal{C}^{\star}_{DPP}(\bt)
\end{align*}

To find a closed-form for the profile likelihood (eq. (\ref{eq:loglik-profile}), we take the derivative of $\mathcal{C}(\bt,\nu)$ with respect to $\nu$, to find:
\begin{equation*}
\frac{\partial}{\partial \nu }\mathcal{C}(\bt,\nu) =  k - \Tr ((\bI + e^\nu \bL)^{-1} e^\nu \bL)
\end{equation*}
where we recognise the saddlepoint equation in yet another form.

Equating the above to 0, we obtain:
\begin{equation}
  \label{eq:profile-likelihood-expl}
  \mathcal{C}^{\star}_{DPP}(\bt) =  \log \det(\bL(\bt)_{\X}) + \nus(\bt) k - \log \det \left( \bI + e^{\nus} \bL(\bt) \right)
\end{equation}

From eq. \eqref{ eq:ESP-saddlepoint} we know that:
\begin{equation}
  \label{eq:saddlepoint-inference}
    \log e_{k}(\bl) = \sum_{i=1}^{n} \left( \log(1+\lambda_{i}e^{\nus}) \right) - k\nus - \frac{1}{2} \left(  \log \left(\sum_{i=1}^{n} \frac{\lambda_{i}}{1+\lambda_{i}e^{\nus}} \right) + \log(2\pi) \right) + O(n^{-1})
\end{equation}
which tells us that:
\begin{equation}
  \label{eq:approx-cost-inference}
  \mathcal{C}^{\star}_{DPP}(\bt) = \mathcal{C}_{kDPP}(\bt) + O(\log(n) + n^{-1}) + constant
\end{equation}
where the term in $O(\log(n))$ comes from the second derivative of $\psi$ in \eqref{ eq:ESP-saddlepoint} and is expected to be small compared to $\mathcal{C}$. 
A full formal argument showing convergence of  $ \hat{\bt}_{DPP}$  to $\hat{\bt}_{k-DPP} $ is complicated, and amounts to showing that the $O(\log(n) + n^{-1})$ term is constant in a relevant region around $\hat{\bt}_{k-DPP}$. 
Informally, however, what happens is quite clear: the two cost functions are close (up to a vertical shift), and if they are sufficiently well-behaved (as a function of $\bt$), then we expect $ \hat{\bt}_{DPP} \approx   \hat{\bt}_{k-DPP} $. We verify this conjecture in a numerical example in section \ref{sec:inference-numerical-res}.

\section{Algorithms and numerical results}
\label{sec:alg-results}

The results above are interesting theoretically, but can also be used in practice to develop algorithms that compute (approximate) ESPs, sample diagonal $k$-DPPs, and compute inclusion probabilities. We find empirically that although approximate, they are much better behaved numerically than their nominally exact counterpart. 

\subsection{Computing ESPs}
\label{sec:computing-esps}

The algorithm given in \cite{KuleszaTaskar:DPPsforML} (alg. 7, p. 60) for computing ESPs of all orders is fast \footnote{Their algorithm runs in $O(n^{2})$, like ours, although theirs is faster in practice.   } but prone to numerical problems when $n$ is large, which is not completely surprising given that ESPs can vary over dozens of orders of magnitude. We find that the saddlepoint approximation given in eq. (\ref{ eq:ESP-saddlepoint}) is more practical, especially since it is naturally computed on a logarithmic scale, a perk exact algorithms do not share. 
To compute eq. (\ref{ eq:ESP-saddlepoint}), one needs to solve the saddlepoint equation (eq. (\ref{eq:nu_implicit})) for $\nu$. Newton's algorithm can be used (alg. \ref{alg:newton-saddlepoint}), but it needs appropriate initialisation or it may not converge (when it does converge, it does so very fast).
In our implementation, an initial guess for $\nu$ is found by linearising the saddlepoint equation for small $\nu$ (small $k$), and large $\nu$ (large $k$).
In small $\nu$, we have:
\begin{equation*}
    \sum_i \frac{\lambda_i e^\nu}{1+\lambda_ie^\nu} \approx e^\nu \sum \lambda_{i}  
  \end{equation*}
  so that for $k$ small, we may approximate $\nu$ as:
  \begin{equation}
    \label{eq:nu_k_small}
    \nu \approx \log k - \log \sum \lambda_{i}
  \end{equation}
  In large $k$ we find:
\begin{equation*}
    \sum_i \frac{\lambda_i e^\nu}{1+\lambda_ie^\nu} \approx n - e^{-\nu} \sum \frac{1}{\lambda_{i}}
  \end{equation*}
  which solving for $\nu$ results in: 
  \begin{equation}
    \label{eq:nu_k_large}
    \nu \approx \log(n-k) - \log \sum \lambda_{i}^{-1}
  \end{equation}
  We use the first guess for $k \leq \frac{n}{2}$ and the second otherwise, with good results. Interestingly, (\ref{eq:nu_k_small}) and (\ref{eq:nu_k_large}) can be used to find the worst-case relative error of the saddlepoint approximation, which is about 10\%, a figure we verify in practice for all but the smallest $n$. The saddlepoint approximation is at its worst far out in the tails, that is, for $k=1$ and $k=n-1$. Recall that $e_{1}(\bl) = \sum \lambda_{i}$. We inject (\ref{eq:nu_k_small}) into (\ref{ eq:ESP-saddlepoint}) and linearise to find:
  \begin{equation}
    \label{eq:error-approx-esp}
    \frac{1}{\sqrt{2\pi \psi''(\nus}) }\exp \left(\sum_{i=1}^{n}( \log(1+\lambda_{i}e^{\nus})) - k\nus \right) \approx \frac{1}{\sqrt{2\pi}} \exp (1) \sum \lambda_{i}  
  \end{equation}
so that the relative error is about $\frac{\exp (1)}{\sqrt{2\pi}} \approx 1.08$. A similar calculation for $k = n-1$ yields the same figure. 

To compute all ESPs, it is useful to begin at $k=1$ and then ``warm-start'' the optimisation rather than always use the same initial condition. The procedure is outlined in algorithm \ref{alg:all-esps}. 

\begin{algorithm}
  \caption{Solving the saddlepoint equation}
  \label{alg:newton-saddlepoint}
  Input: eigenvalues $\bl$, set size $k$, initial guess $\nu_{0}$, tolerance $\epsilon$
  
  \begin{algorithmic}
    \Procedure{solve}{$\bl$,$k$,$\nu_{0}$}
    \State $\nu \gets \nu_0$
\While {$|\psi'(\nu) - k| < \epsilon $}
    \State $\nu \gets \nu - \frac{\left(\psi'(\nu) - k\right)}{\psi''(\nu)} $ (eq. \eqref{eq:cumulants-first-derivative} and \eqref{eq:cumulants-second-derivative}). 
    \EndWhile
    \EndProcedure
\end{algorithmic}
Return $\nu$
\end{algorithm}

\begin{algorithm}
  \caption{Computing all (log-) ESPs using a saddlepoint approximation}
  \label{alg:all-esps}
  Input: eigenvalues $\bl$ (vector of length $n$),
  Set initial guess to $\nu= - \log \sum \lambda_{i}$. 
  \begin{algorithmic}
    \For { $k \gets 1 \ldots (n-1) $ }
    \State $\nu \gets \text{solve}(\bl,k,\nu)$
    \State $\log e_{k} \gets - \frac{1}{2} \log\left(2 \pi \psi''(\nu) \right) + \sum_{i=1}^{n} \log(1+\lambda_{i}e^\nu) - k\nu  $
\EndFor
\end{algorithmic}
Return $\log e_{1} \ldots \log e_{n-1}$
\end{algorithm}

\subsection{Computing inclusion probabilities}
\label{sec:inclusion-prob}

\subsubsection{In diagonal $k$-DPPs}
\label{sec:inclusion-prob-diag}

\begin{algorithm}
  \caption{First order inclusion probabilities in diagonal $k$-DPPs: basic estimate}
  \label{alg:first-order-diag-simple}
  Input: eigenvalues $\bl$, set size $k$.

  \begin{algorithmic}
    \Procedure{diag-basic}{$\bl$,$k$}
        \State $\nu \gets \text{solve}(\bl,k,\nu_{0})$ (Solve for saddle point)
    \For {$i \in 1 \ldots n$} 
    \State     $\tilde{\pi}_{i} \gets \frac{e^{\nu}\lambda_{i}}{1+e^{\nu}\lambda_{i}}$
    \EndFor
    \EndProcedure
\end{algorithmic}
Return $\tilde{\pi_{1}} \ldots \tilde{\pi_{n}} $
\end{algorithm}

\begin{algorithm}
  \caption{First order inclusion probabilities in diagonal $k$-DPPs: corrected estimate}
  \label{alg:first-order-diag-corrected}
  Input: eigenvalues $\bl$, set size $k$.

  \begin{algorithmic}
    \Procedure{diag-corrected}{$\bl$,$k$}
    \State $\nu \gets \text{solve}(\bl,k,\nu_{0})$ (Solve for saddle point)
    \State $\psib'' \gets  \frac{1}{n} \sum \frac{\lambda_{i}e^\nu }{(1+\lambda_{i}e^\nu)^{2}}$
    \State $\psib^{(3)} \gets \frac{1}{n} \sum \frac{\lambda_{i}e^\nu(1-\lambda_{i}e^{\nu}) }{(1+\lambda_{i}e^\nu)^{3}}$

    \For {$i \in 1 \ldots n$}
    \State $\psib'_{i} \gets \frac{e^{nu}\lambda_{i}}{1+e^{nu}\lambda_{i}}$
    \State $\psib''_{i} \gets \frac{\lambda_{i}e^\nu }{(1+\lambda_{i}e^\nu)^{2}}$
    \State $\nu_{1} \gets \frac{1-\psib'_{i}}{\psib''}$
     \State     $g \gets -\frac{\nu_{1}^{2}}{2}\psib''
      - \frac{1}{ 2 \psib''}  \left( \psib^{(3)}  \nu_{1} - m \psib_{\ba}'' \right)$

    \State     $\tilde{\pi}_{i} \gets \frac{e^{\nu}\lambda_{i}}{1+e^{\nu}\lambda_{i}}(1+\frac{g}{m})$
    \EndFor
    \EndProcedure
\end{algorithmic}
Return $\tilde{\pi_{1}} \ldots \tilde{\pi_{n}} $
\end{algorithm}

\begin{algorithm}
  \caption{First order inclusion probabilities in general $k$-DPPs}
  \label{alg:first-order-general}
  Input: L-ensemble $\bL$, set size $k$

  \begin{algorithmic}
      \State  $\bU \gets \text{eigenvectors}(\bL)$, $\bl \gets \text{eigenvalues}(\bL)$
      \State $\tilde{\bm{\pi}} \gets \text{diag-simple}(\bl,k)$ or $\tilde{\bm{\pi}} \gets \text{diag-corrected}(\bl,k)$
      \For {$i \in 1 \ldots n$}
      \State $\tilde{p}_{i} \gets \sum_{j=1}^{n} U_{ij}^{2} \pi_{j}$
      \EndFor
    \end{algorithmic}
Return $\tilde{p}_{1}  \ldots \tilde{p}_{1} $
  \end{algorithm}

\begin{algorithm}
  \caption{High order inclusion probabilities in general $k$-DPPs}
  \label{alg:high-order-general}
  Input: L-ensemble $\bL$, subset $\ba$, set size $k$

  \begin{algorithmic}
     
      \State  $\bl \gets \text{eigenvalues}(\bL)$.
      \State $\nu \gets \text{solve}(\bl,k,\nu_{0})$
      \State $\tilde{p}_{\ba} \gets \det \left( (\bI + e^{\nu}\bL)^{-1} e^{\nu} \bL  \right)_{\ba}$
      \State (Optional: compute correction
      \State $m \gets |\ba|$ (size of subset)
      \State $\tilde{\bm{\pi}} \gets \text{diag-simple}(\bl,k)$
      \State $v \gets e_{m}(\tilde{\bm{\pi}})$ 
      \State $\tilde{p}_{\ba} \gets \tilde{p} \times {k \choose m}v^{-1}$)
    \end{algorithmic}
Return $\tilde{p}_{\ba}$
\end{algorithm}

Computing inclusion probabilities boils down to an application of the asymptotic formula developed in section \ref{sec:asymptotics-inclusion-proof}. Depending on the accuracy required, the $O(\epsilon)$ in (\ref{eq:incl-prob-improved}) may or may not be needed. Computing the terms in (\ref{eq:incl-prob-improved}) requires solving the saddlepoint equation, and computing $\psi(\nus)$ and up to three derivatives, which is $O(n)$ work in total. $\psi'$ and $\psi''$ are by-products of the Newton iteration, and only $\psi^{(3)}$ must be computed from scratch. All first order inclusion probabilities can be computed jointly based on these quantities, in $O(n)$ time.
Algorithm \ref{alg:first-order-diag-simple} computes the uncorrected estimate, and alg. \ref{alg:first-order-diag-corrected} the corrected estimate (the latter may be hard to understand without a thorough look at section \ref{sec:asymptotics-inclusion-proof}). 

\subsubsection{In general $k$-DPPs}
\label{sec:inclusion-prob-general}

Computing (approximate) inclusion probabilities in general $k$-DPPs is an application of eq. (\ref{eq:inclusion-prob-general}): first compute the inclusion probabilities for the eigenfunctions, then apply (\ref{eq:inclusion-prob-general}). For first-order inclusion probabilities, and under the $O(\frac{1}{n})$ approximation, this boils down to computing
\begin{equation}
  \label{eq:first-order-as-diag}
\diag \left( e^{\nus}( e^{\nus} \bL + \bI)^{-1}  \bL \right)  
\end{equation}
where $\nus$ solves the saddlepoint equation for the appropriate value of $k$. $\nus$ depends on the eigenvalues of $\bL$, so the most straightforward way of computing (\ref{eq:first-order-as-diag}) is via an eigendecomposition of $\bL$, after which one obtains (\ref{eq:first-order-as-diag}) from:
\begin{equation}
  \label{eq:first-order-from-eigenvalues}
  \left( e^{\nus} (e^{\nus} \bL + \bI)^{-1}  \bL \right)_{ii}  =
  \sum_{j=1}^{n} \frac{ e^{\nus} \lambda_{j}}{1+e^{\nus} \lambda_{j}} U_{ij}^{2}
\end{equation}
where $U_{ij}$ is the $j$-th eigenvector of $\bL$ evaluated at index $i$. 
In most realistic problems the dominant cost by far is the $O(n^{3})$ eigendecomposition. If $\bL$ is sparse, or if matrix-vector products $\bL\mathbf{v}$ can be computed using fast algorithms, the cost can be significantly reduced using a variety of techniques. For example, under sparse $\bL$ the Cholesky decomposition may still be relatively cheap, and the Takahashi equations \citep{RueHeld:GaussMarkovRandomFields} can be used to obtain diagonal elements in (\ref{eq:first-order-as-diag})  and solve the saddlepoint equation.  
Alg. \ref{alg:first-order-general} computes first-order probabilities, and alg. \ref{alg:high-order-general} higher order inclusion probabilities. The optional correction used in alg. \ref{alg:high-order-general} is explained in section \ref{sec:corrected_approx}. 
\subsection{Sampling}
\label{sec:sampling-$k$-DPPs}

There already is a large literature on sampling from ($k$-)DPPs (see e.g., \citep{Li:EffSamplingDPP,Gautier:ZonotopeSamplingDPPs} and references therein). Our goal here is only to show how our methods can be used to modify the algorithms given in \cite{KuleszaTaskar:DPPsforML} to improve numerical stability. 
To sample a $k$-DPP, we follow the two-step strategy of \cite{KuleszaTaskar:DPPsforML}, which derives from the mixture interpretation explained in section \ref{sec:kDPPs}. We first sample a set of eigenvectors, picking $k$ of them using a diagonal $k$-DPP, then sample from the projection DPP formed from the eigenvectors we selected.

\subsubsection{Sampling from a diagonal $k$-DPP}
\label{sec:sampling-diag}

The first part requires sampling from a diagonal $k$-DPP.  For that task, \cite{KuleszaTaskar:DPPsforML} give an algorithm that they justify using a recursive argument, but a more intuitive explanation can be found. Thinking of the $k$-DPP as sampling a binary string $\bz = z_{1} \ldots z_{n}$, we run through the elements one by one, sampling according to $p(z_{t} | z_{1} \ldots z_{t-1} )$. It is straightforward to show that in a diagonal $k$-DPP,  $z_{t} | z_{1} \ldots z_{t-1}$ has a sufficient statistic: $p(z_{t} | z_{1} \ldots z_{t-1} ) = p(z_{t} | \sum_{i=1}^{t-1} z_{i} )$. This occurs because $z_{t} \ldots z_{n} | z_{1} \ldots z_{t-1}$ is a diagonal $(k-s)$-DPP, where $s = \sum_{i=1}^{t-1} z_{i} $. Thus, $p(z_{t} = 1 | z_{1} \ldots z_{t-1} )$  is the inclusion probability for item $t$ in a diagonal $(k-s)$-DPP, and we can use our approximations to compute that probability.

\begin{algorithm}
  \caption{Sampling from a diagonal $k$-DPP}
  \label{alg:sampling-diag}
  Input: eigenvalues $\bl$ (vector of length $n$), integer $k$ (set size).
  Init $s=0,t=1$
  \begin{algorithmic}
    \While { $t \leq n, s < k $ }
    \State Compute $\pi_{t}$, inclusion probability in a diagonal $(k-s)$-DPP with eigenvalues $\lambda_{t} \ldots \lambda_{n}$, using eq. (\ref{eq:incl-prob-improved}).
    \State Set $z_{t}$ to $1$ with probability $\pi_{t}$
    \State $s \gets \sum_{i=1}^t z_{i} $
    \State $t \gets t + 1$
    \EndWhile

\end{algorithmic}
    Return $\bz$, the inclusion vector.
\end{algorithm}

We state the basic algorithm in alg. \ref{alg:sampling-diag}, but many refinements can be made for speed. In particular, computing the approximation requires solving the saddlepoint equation, and warm-starting should be used. Beyond that, given that the cost of sampling from a $k$-DPP is mostly dominated by the eigenvalue decomposition and by the step where a projection DPP is sampled, it is not worth spending too much time optimising the diagonal step. 

\subsection{Sampling from a projection DPP}

Once we have obtained $k$ eigenvectors, we can form the projection kernel $\mathbf{U}_{:,\Y}\mathbf{U}_{\Y,:}^\top$ and use any algorithm that samples a projection DPP. There are several options in the literature, but one that is both fast and particularly easy to implement is described in \citep{Tremblay:OptimAlgSampleDPPs}, and that is the one we use here in our simulations.

\section{Empirical results}
\label{sec:empirical_results}

We report here the accuracy of our approximations in some tests and simulations. We examine briefly the quality of the approximation for ESPs, then inclusion probabilities in diagonal $k$-DPPs, and finally inclusion probabilities in general $k$-DPPs. 

\subsection{Approximation of Elementary Symmetric Polynomials}
\label{sec:approx_esps}

The approximation to ESPs given in eq. (\ref{ eq:ESP-saddlepoint}) is nothing more than a saddlepoint approximation for sums of Bernoulli variables, so it would be surprising if it did not work as advertised. Nonetheless it is interesting to see how good the approximation is, and that the figure we give in section \ref{sec:computing-esps} for a maximum error of 10\% (for $k=1$ and $k=n-1$) is verified in practice. It also serves to illustrate the better numerical behaviour of the approximation compared to the (nominally exact) summation algorithm.

Figure \ref{fig:ESP_1toN} and \ref{fig:ESP_exp_min_1_to_n} show results obtained on two deterministic sequences, $ \lambda_{i,n} = i $ and $\lambda_{i,n} = e^{-i}$ (for three different values of $n$). The approximation is excellent even with the second sequence, which does not verify the sufficient condition for convergence ($\sigma^{2}=O(n)$). The summation algorithm overflows in the first case and underflows in the second, while the approximation shows good behaviour. 

To go beyond deterministic sequences, we consider a set of $n$ points in $\R^{2}$, drawn from a unit Gaussian distribution. The $\bL$-matrix is from a squared-exponential kernel,
$L_{ij}=\exp \left(\frac{ -|| x_{i} - x_{j} ||^{2} }{ 2 \tau ^{2}}  \right)$. Here we set $\tau=1$. Figure \ref{fig:rel_error_gaussker} shows the ratio of approximation to true value $\tilde{e}_{k}(\bl)/e_{k}(\bl)$, where $\bl$ are the eigenvalues of $\bL$. 

\begin{figure}
  \includegraphics[width=10cm]{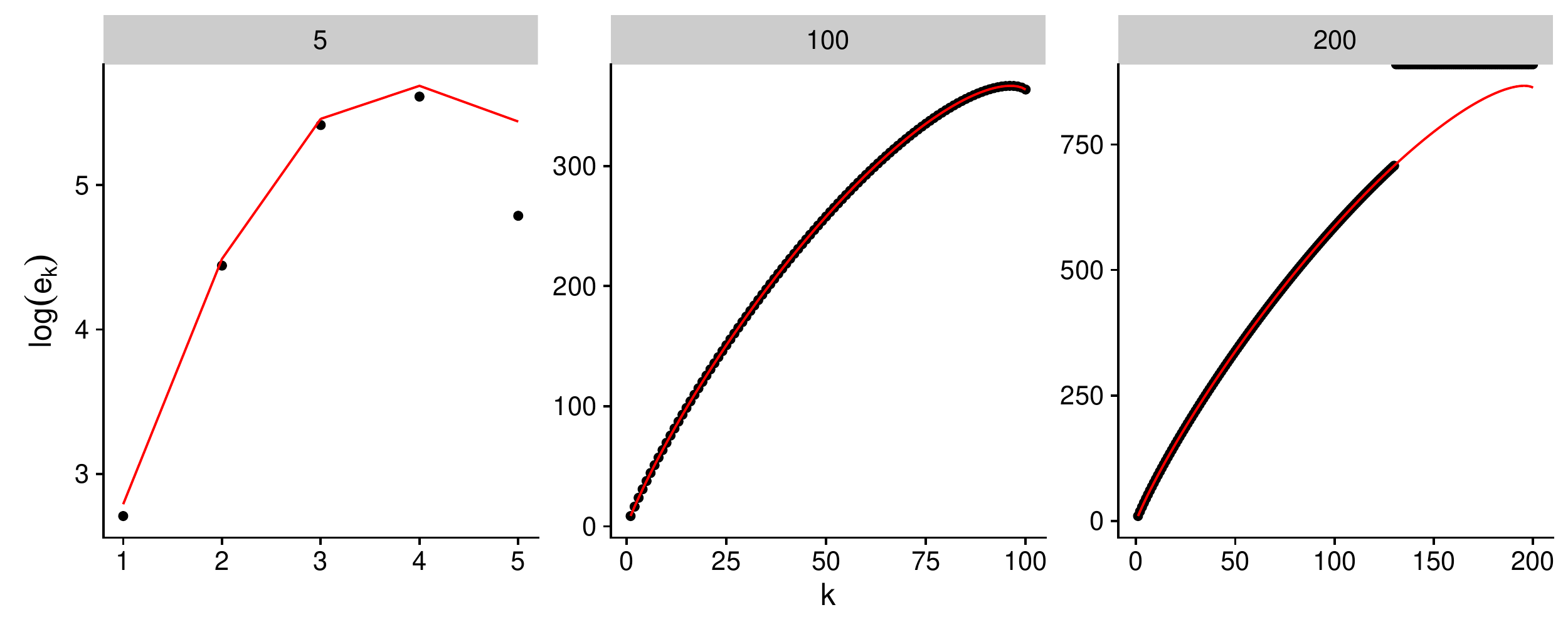}
  \centering
  \caption{Approximation of  (log) ESPs for the sequence $\lambda_i=i , i=1,\ldots, n$. Black dots: numerical results obtained using the ``exact'' summation algorithm. Red line: approximation using eq. (\ref{ eq:ESP-saddlepoint}). We show results for $n$=5, 100 and 200 . Numerical problems are already apparent at $n$=200, with the summation algorithm overflowing at $k \approx 130$. The approximation has no such issues in this case.    }
  \label{fig:ESP_1toN}
\end{figure}

\begin{figure}
  \centering
  \includegraphics[width=10cm]{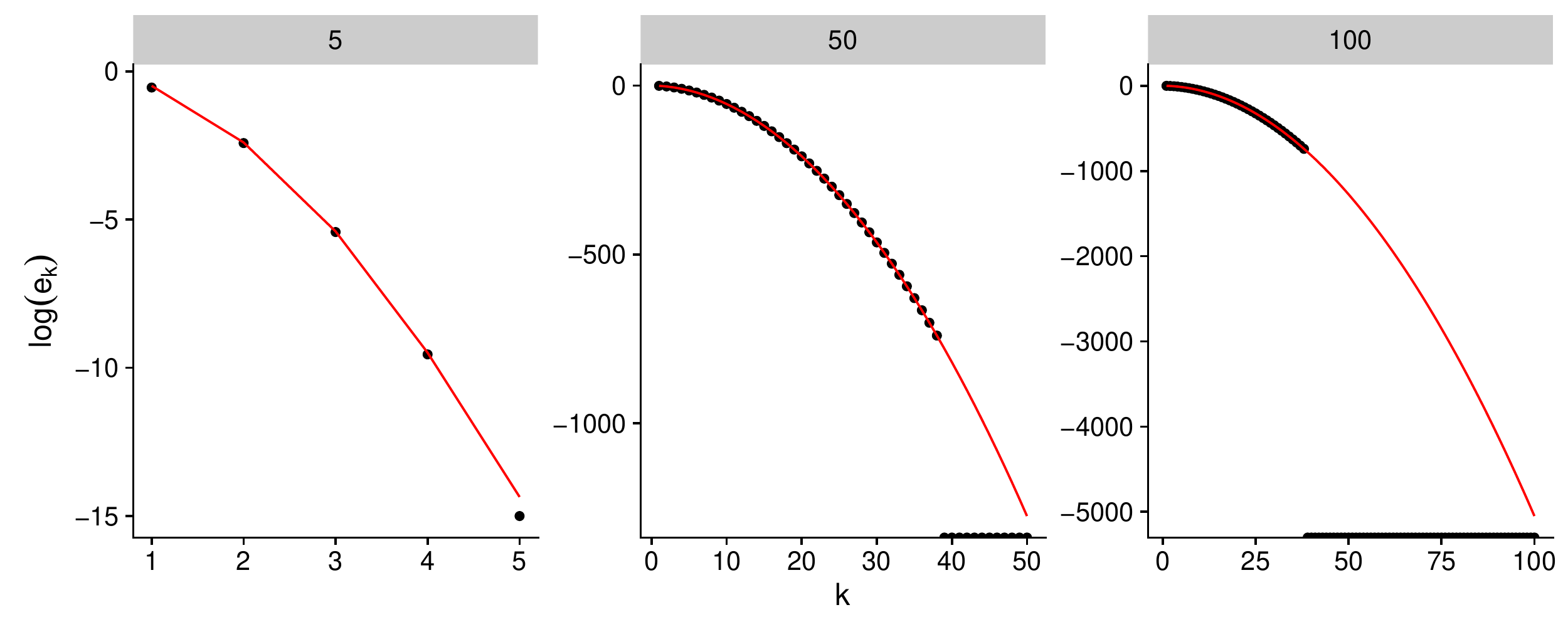}

  \caption{Approximation of (log) ESPs for the sequence $\lambda_i=e^{-i} , i=1,\ldots, n$. Same format as figure \ref{fig:ESP_1toN} }
  \label{fig:ESP_exp_min_1_to_n}
\end{figure}

\begin{figure}
  \centering
  \includegraphics[width=10cm]{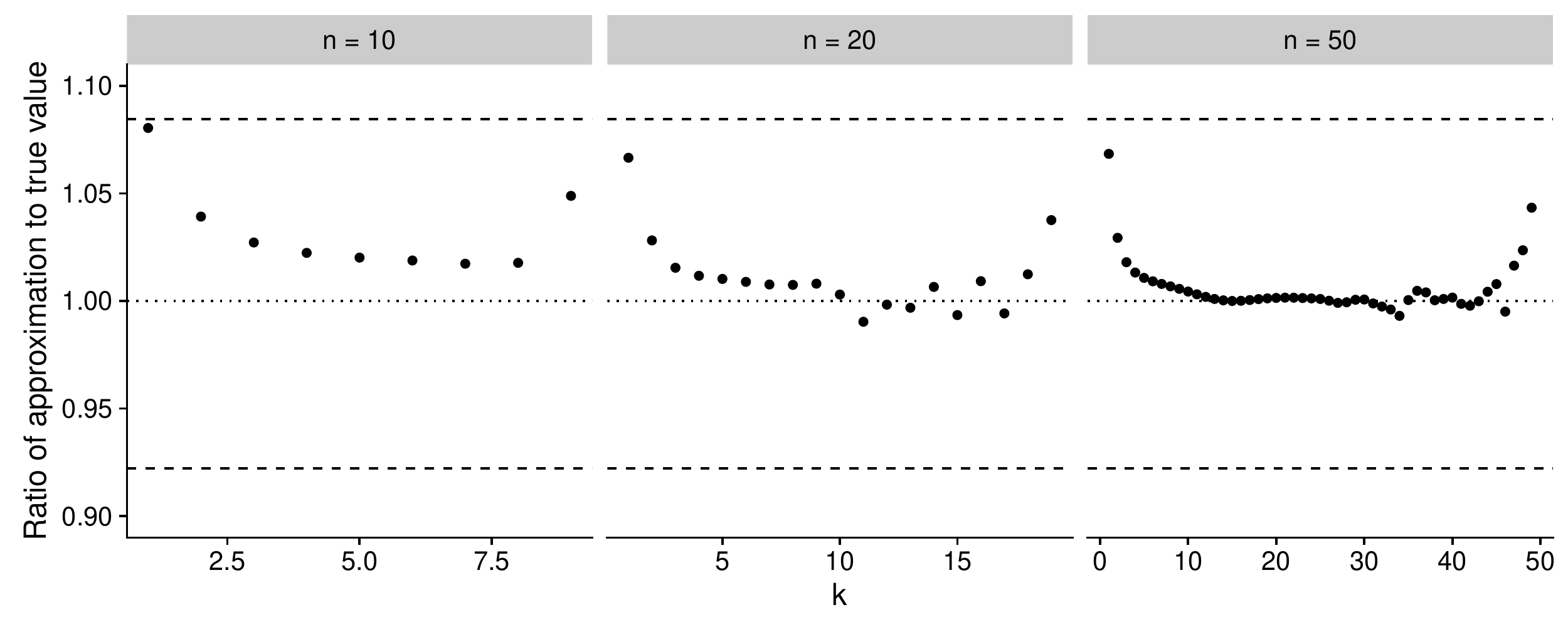}
  
  \caption{Ratio of approximated ESP to true ESP, for an example with $n$ points at random locations and a squared-exponential kernel (see text). From the arguments in section \ref{sec:computing-esps}, we expect a maximum relative error of about 1.08, shown here as upper and lower dashed lines. The central dashed lines corresponds to no error.  }
  \label{fig:rel_error_gaussker}
\end{figure}

\subsection{Approximation of inclusion probabilities}
\label{sec:results-approx-inclusion}

We begin with approximations to inclusion probabilities in diagonal $k$-DPPs. Figure \ref{fig:incl-diag} shows results for a diagonal $k$-DPP with diagonal values $\lambda_{i,n} = \exp(-\frac{i}{10})$, comparing true inclusion probabilities to the $O(n^{-1})$ and $O(n^{-2})$ approximations given by lemma \ref{lm:inclusion-probs}. The approximations are overall excellent, with even the rougher $O(n^{-1})$ approximation becoming practically exact for $n \geq 100$. Note that the conditions for convergence assumed in our theorem do not hold for the sequence in question. 

The $O(n^{-1})$ and $O(n^{-2})$ rates are asymptotic, and it is interesting to verify that they hold in practice. Figure \ref{fig:convergence-rates} does this, in a scenario where the conditions of the theorem hold. For each $n$, the diagonal values $\lambda_{1,n}$ to $\lambda_{n,n}$ are drawn i.i.d. from the uniform distribution on the interval $(1,10)$. We estimate convergence rates via a regression of $\log$ error on $\log n$.

\begin{figure}
  \centering
  \includegraphics[width=10cm]{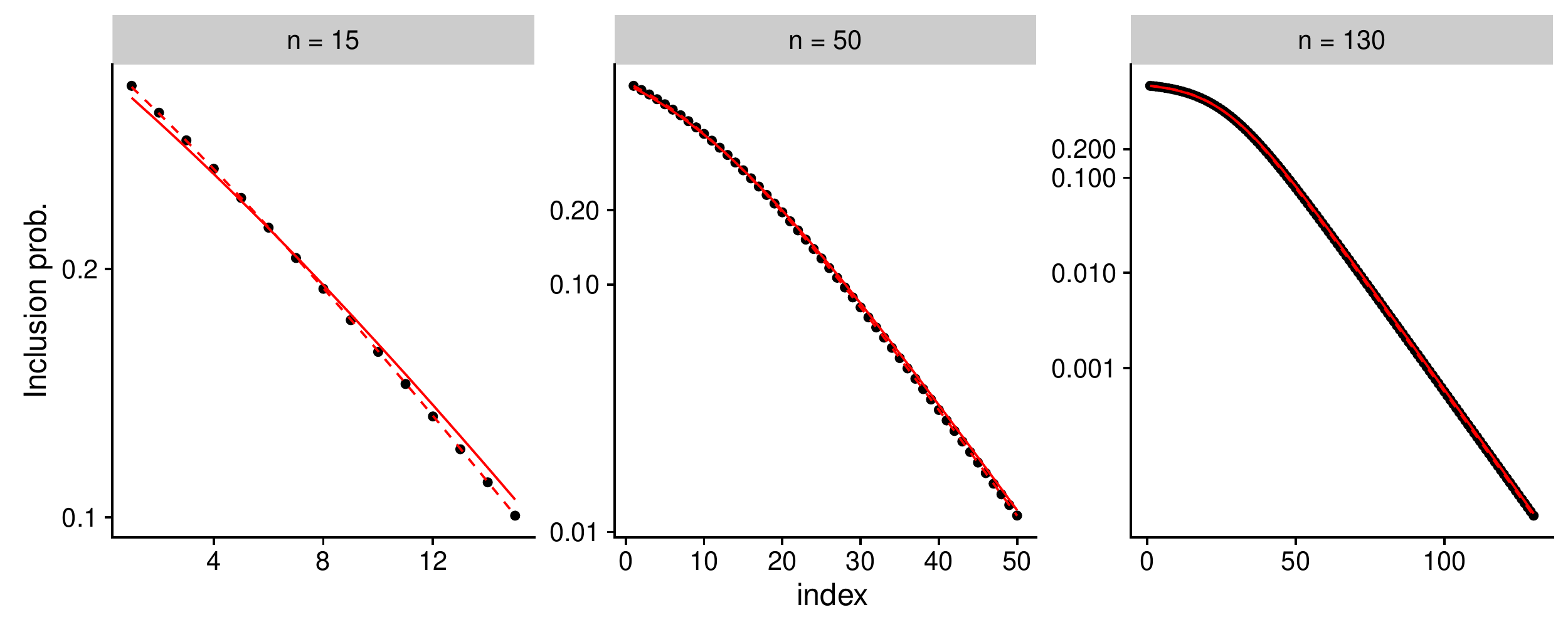}
  \caption{Inclusion probabilities for a diagonal $k$-DPP with diagonal entries $\exp(-\frac{i}{10})$, with $k = n/5$. Black: true values. Red, continuous: $O(n^{-1})$ approximation. Red, dashed, $O(n^{-2})$ approximation (see eq. (\ref{eq:incl-prob-improved})).  }
  \label{fig:incl-diag}
\end{figure}
\begin{figure}
  \includegraphics[width=6cm]{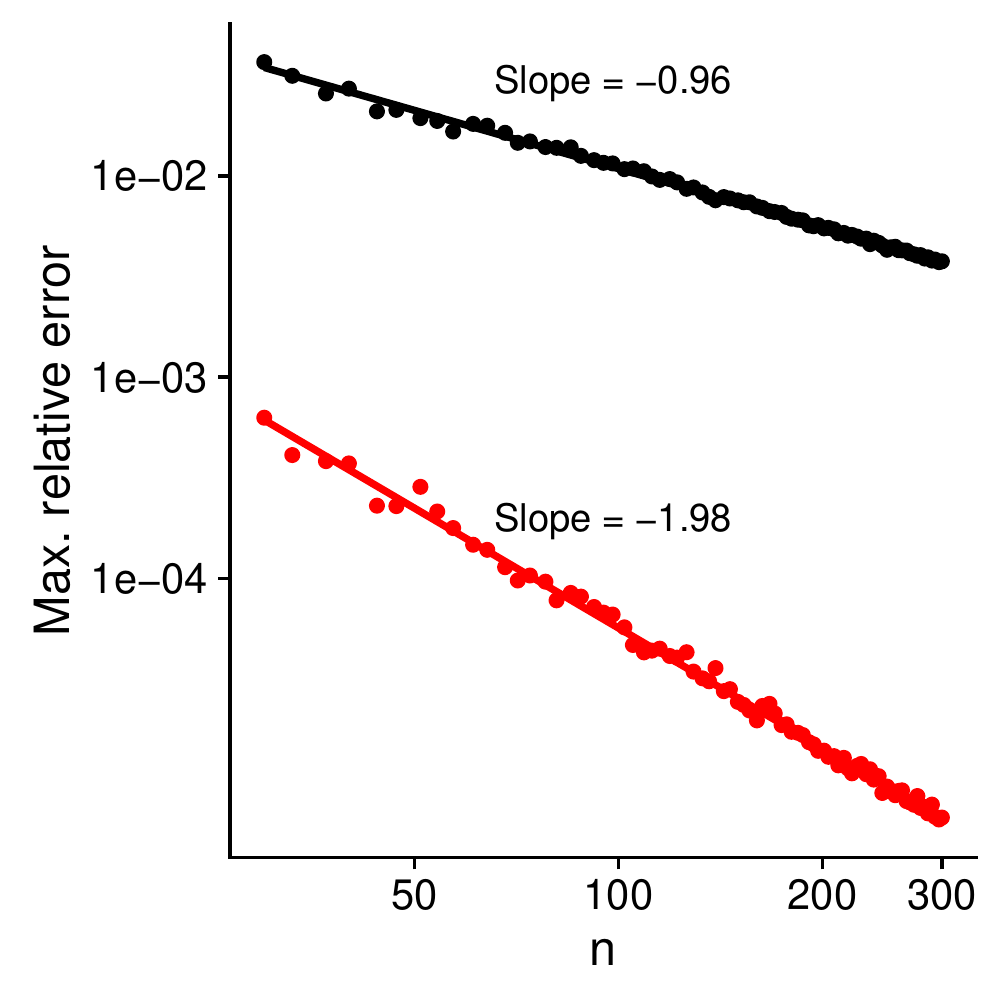}
  \centering
  \caption{Convergence of approximations to diagonal probabilities. We verify the $O(n^{-1})$ and $O(n^{-2})$ rates empirically. See text for details. }
\label{fig:convergence-rates}
\end{figure}

Unsurprisingly given the above, approximating inclusion probabilities in general $k$-DPPs works well too. Figure \ref{fig:incl_prob_gaussker} provides an illustration, using again $n$ points drawn i.i.d. from a Gaussian, as in figure \ref{fig:rel_error_gaussker}. Both approximations work extremely well for realistic values of $n$, and again we stress that this is a case in which the conditions for large-$n$ convergence do not hold (because the eigenvalues decrease too fast). 

\begin{figure}
  \centering
  \includegraphics[width=12cm]{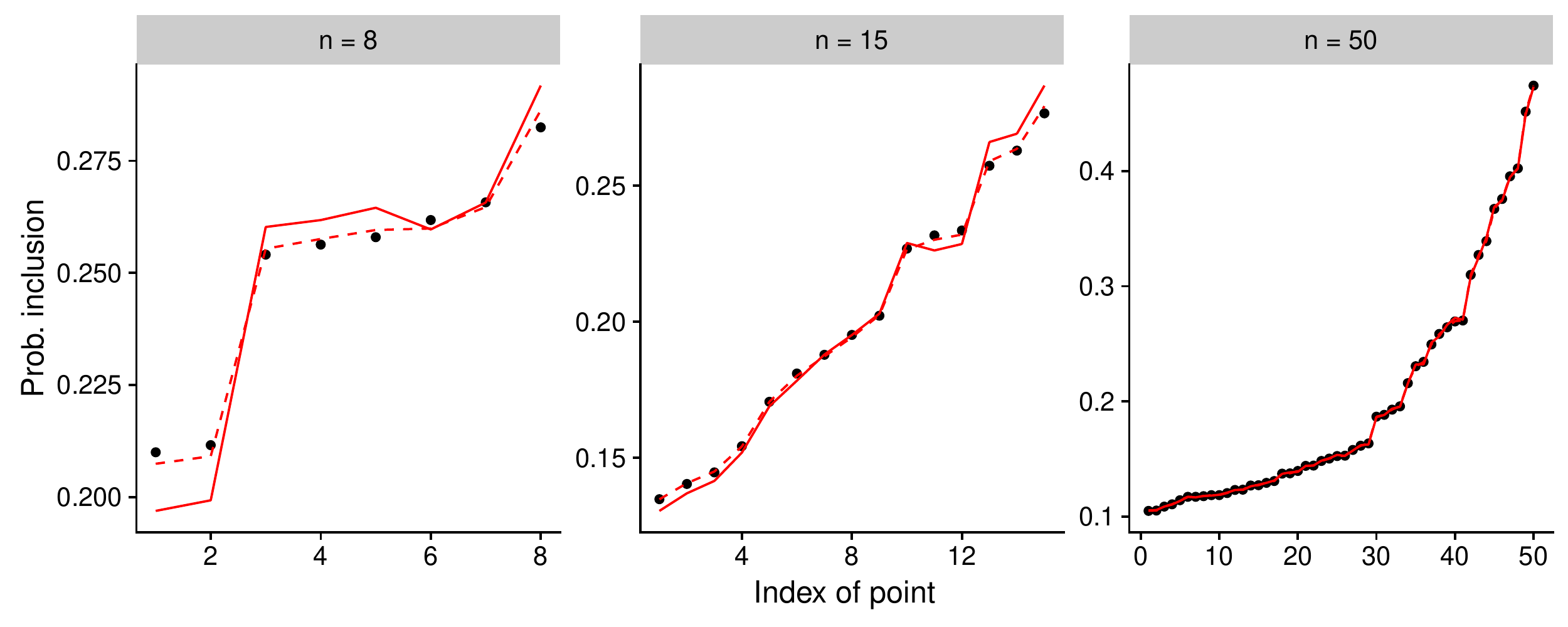}
  \caption{ Approximation to inclusion probabilities in a full $k$-DPP.  The scenario is the same as in figure \ref{fig:rel_error_gaussker}, namely $n$ points drawn  i.i.d. and a squared exponential kernel.  Each point correspond to the inclusion probabilities of point $x_{i}$ in a $k$-DPP. Points have been sorted according to increasing probability of inclusion.  }
  \label{fig:incl_prob_gaussker}
\end{figure}

Theorem \ref{thm:main-result} implies that we can approximate inclusion probabilities for pairs, not just singletons. We show an illustration in figure \ref{fig:incl_bivariate}, where we repeated the above experiment with $n=500$, $k=50$ and $\tau=0.5$. We picked 400 $m$-uples  at random and estimated their true inclusion probability using Monte Carlo \footnote{We used an empirical version of eq. \eqref{eq:incl-prob-mixture}, and 1,500 samples}. We compare the estimated inclusion probability to the $O(n^{-1})$ approximations, and to the corrected probabilities described in Appendix \ref{sec:corrected_approx}. The $O(n^{-1})$ approximation shows a slight bias for high probabilities but is overall very good, and most of the bias is removed by the correction. 

\begin{figure}
  \centering
  \includegraphics[width=10cm]{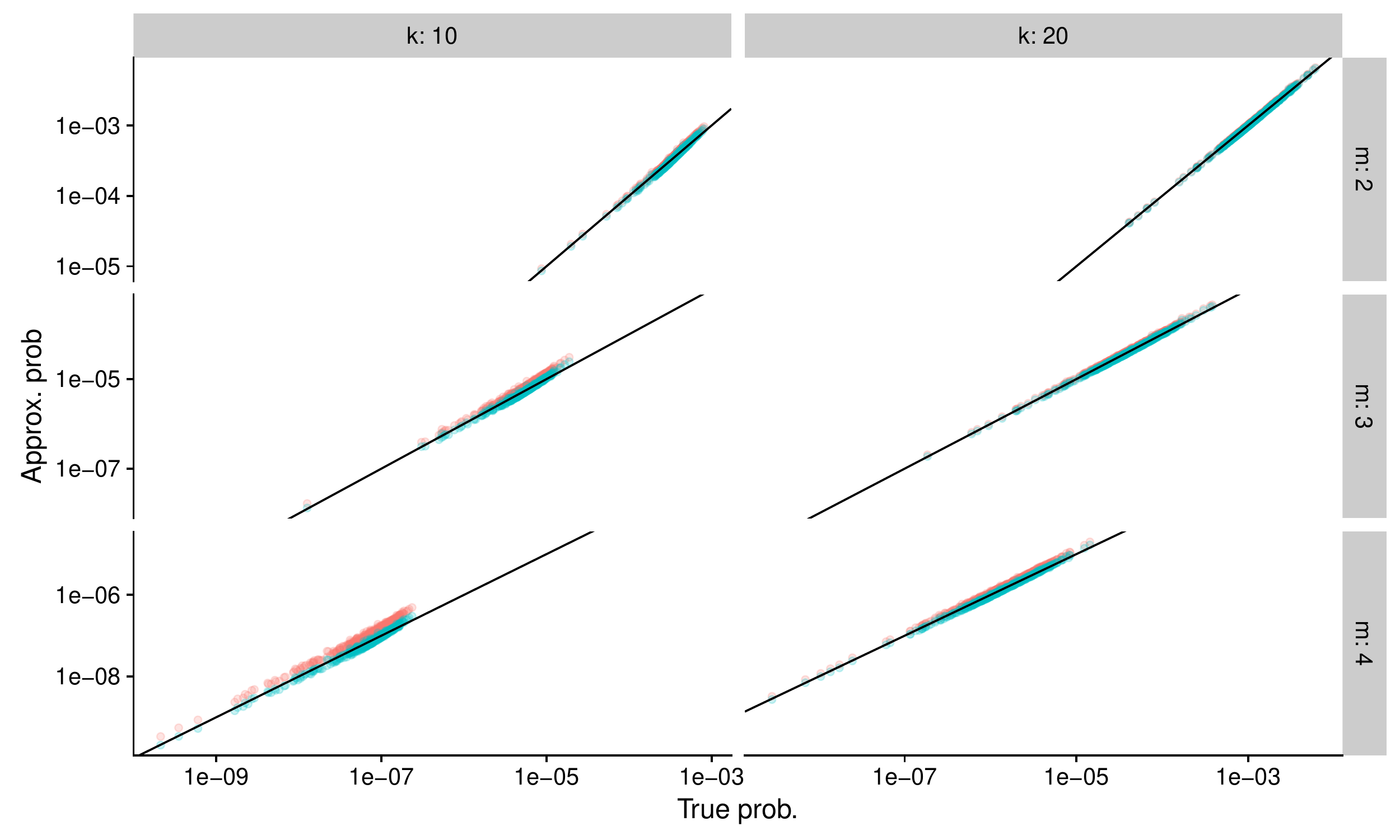}
  \caption{ Approximation to high order inclusion probabilities in a full $k$-DPP.  400 subsets of size $m$ are drawn at random and their 2nd order inclusion probability estimated using Monte Carlo (see text). We compare the estimate to the $O(n^{-1})$ approximation (in red) and the corrected approximation (in blue), see App. \ref{sec:corrected_approx}.   }
  \label{fig:incl_bivariate}
\end{figure}

\subsection{Inference}
\label{sec:inference-numerical-res}

The goal of this section is to illustrate the claims of section \ref{sec:csq-for-inference}, namely that $k$-DPPs and DPPs have equivalent ML estimators (when used as statistical models).

We again used the same setup as in the previous section: $n$ points drawn  i.i.d.  from a 2D Gaussian, with a subset $\X$ of size $k$ drawn from a k-DPP. Contrary to the previous sections, however, the objective here is to infer something about the L-ensemble given $\X$. We use two statistical models:
\begin{enumerate}
\item That $\X$ is drawn from a $k$-DPP with L-ensemble $L_{ij}=\exp \left(-\frac{ || x_{i} - x_{j} ||^{2} }{ 2 \tau ^{2}}  \right)$ (for an unknown value of $\tau$).
\item That $\X$ is drawn from a DPP with L-ensemble $\tilde{L}_{ij}= e^{\nu} \exp \left( -\frac{ || x_{i} - x_{j} ||^{2} }{ 2 \tau ^{2}}  \right)$ (for an unknown value of $\tau$ and $\nu$).
\end{enumerate}

In figure \ref{fig:log-likelihoods} we show the log-likelihood of the $k$-DPP model as a function of $\tau$, along with the profile log-likelihood of the DPP model ($\mathcal{C}^{\star}$, see eq. \eqref{eq:approx-cost-inference}). The maximum likelihood estimates of $\tau$ are the argmax of these curves, and as predicted they are extremely close.

\begin{figure}
  \centering
  \includegraphics[width=12cm]{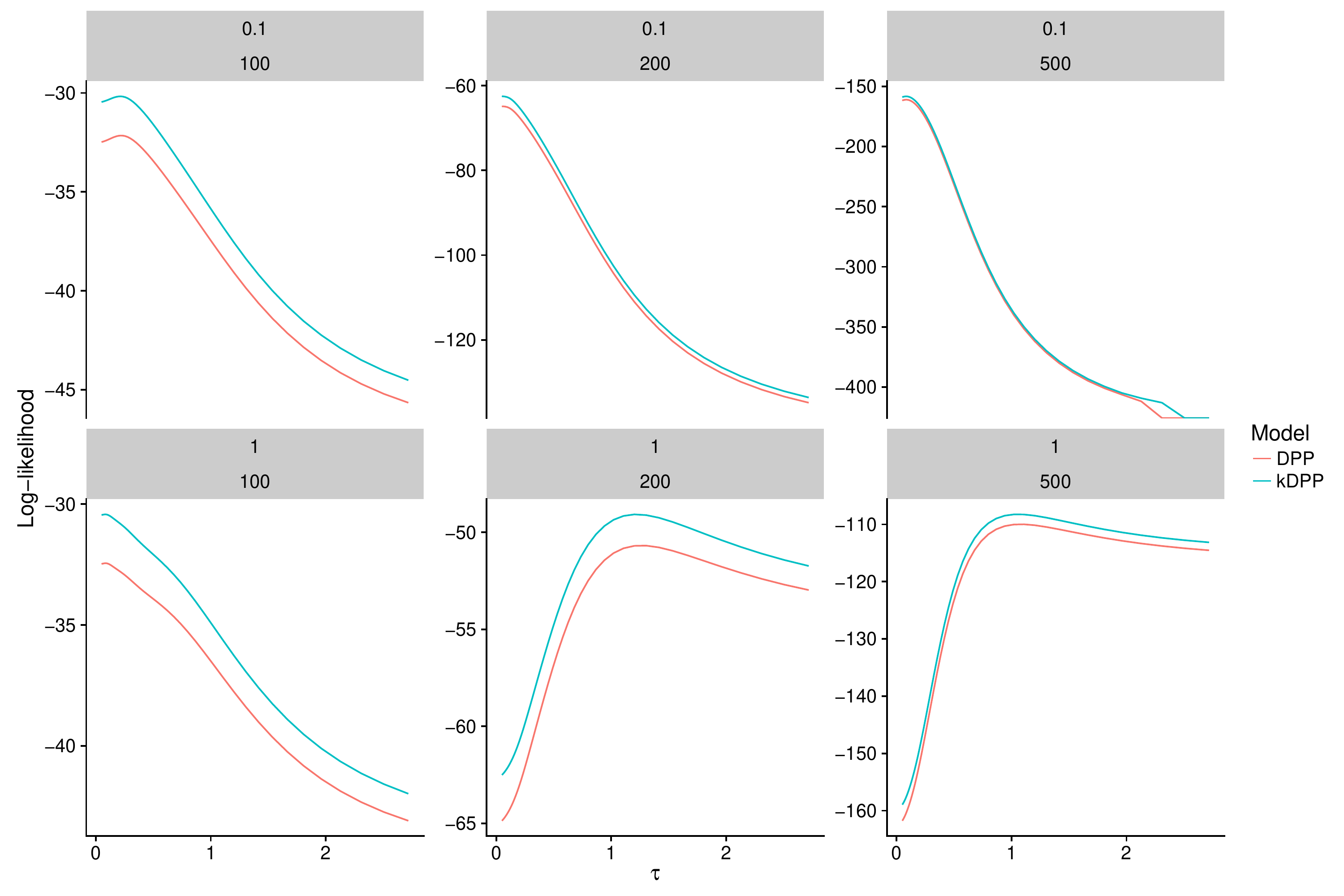}

  \caption{Log-likelihoods of a $k$-DPP and DPP model, for various values of $n$ (100 to 500) and true parameter $\tau$ (0.1 and 1). In blue, the $k$-DPP likelihood, red, the DPP profile log-likelihood.  In all cases, $k=\frac{n}{10}$}
    \label{fig:log-likelihoods}

\end{figure}

\section{Discussion}
\label{sec:discussion}

We have shown that $k$-DPPs are for practical purposes largely equivalent to DPPs, so that one can sample from a $k$-DPP, pretend that the realisation actually came from a matched DPP, and expect no major damage. Corrections to the inclusion probabilities come at little extra cost and increase the accuracy enough so that the approximations can be used with very small $n$. The saddlepoint approximation can be used to compute ESPs as well, and if more accuracy is needed we suggest including further Edgeworth terms.
The remaining hurdle is to develop appropriate algorithms that estimate the relevant functions of the L-ensemble, to remove the need for an eigenvalue decomposition. We hope to develop such methods in future work.

\section*{Acknowledgements}
This work benefited from funding from ANR GenGP (ANR-16-CE23-0008), LabEx PERSYVAL-Lab (ANR-11-LABX-0025-01), Grenoble Data Institute  (ANR-15-IDEX-02), CNRS PEPS I3A (Project RW4SPEC), and LIA CNRS/Melbourne Univ Geodesic. 

\bibliographystyle{imsart-nameyear}
\bibliography{../bibtex/ref_dpp} 

\appendix

\section{Appendix}
\label{sec:appendix}

\subsection{Proof of lemma \ref{lm:marginal-kernel-dpp}}
\label{sec:proof-kernel-proj}

Here we prove lemma \ref{lm:marginal-kernel-dpp}, which states that the inclusion kernel of a projection DPP equals the L-ensemble.
We need to compute the probability that $\ba \subseteq \X$, where $\X$ is a sample from a $k$-DPP with L-ensemble $\bL = \bU \bU^{\top}$, and $\bU$ is a $n \times k$ matrix such that $\bU^{\top} \bU = \bI$. It is important here that we are sampling sets of size $k$ from an L-ensemble of rank $k$. As elsewhere we note $|\ba| = m \leq k$.

We need the following well-known result on determinants of bordered matrices:
\begin{equation}
  \label{eq:det-bordered}
  \det \left[
    \begin{array}{c c}
      \mathbf{A} & \mathbf{b} \\
      \mathbf{b^\top} & c 
    \end{array}
  \right]
    = \left( \det \mathbf{A}  \right) \left(c - \mathbf{b}^\top \mathbf{A} ^{-1} \mathbf{b} \right)
\end{equation}

We also make use of the following result, which lets us perform partial sums in determinants.

\begin{align*}
  \sum_{i=1}^{n} \left(\bL_{i,i} - \bL_{i,\X} \bL_{\X}^{-1} \bL_{\X,i}\right)
  &= \Tr \bL - \sum_{i} \Tr \left\{ \bL_{\X}^{-1} \bL_{\X,i}\bL_{i,\X} \right\}  \\
  &= k -  \Tr \left\{  \bL_{\X}^{-1} \sum_{i} \bL_{\X,i}\bL_{i,\X} \right\} \\
  &= k - \Tr \left\{ \left( \bU_{\X,:} \bU^{\top}_{:,\X} \right)^{-1} \sum_{i} \left( \bU_{\X,:} \bU^{\top}_{:,i} \right)
    \left( \bU_{i,:} \bU^{\top}_{:,\X} \right) \right\} \\
  &= k - \Tr \left\{ \left( \bU_{\X,:} \bU^{\top}_{:,\X} \right)^{-1} \left( \bU_{\X,:}  \bU^{\top}_{:,\X} \right) \right\} \\
  &= k - |\X| \numberthis   \label{eq:removing-item-proj}
\end{align*}

To simplify what follows, we change the setting a bit and look at \emph{ordered sets}: we sample $\X$ from a DPP, give it a random order (one of $k!$), and thus obtain a vector $\mathbf{x}$. Instead of computing $p(\ba \subseteq \X)$, we compute $p(x_{1} = \alpha_{1}, x_{2} = \alpha_{2}, \ldots, x_{m} = \alpha_{m})$, for one particular ordering of $\ba$. The two probabilities are related: \[ p(x_{1} = \alpha_{1}, x_{2} = \alpha_{2}, \ldots, x_{m} = \alpha_{m}) = p(\ba \subseteq \X) \frac{(k-m)!}{k!} \]

Further, note that the probability mass function for $\mathbf{x}$ is just $p(\mathbf{x}) = \frac{p(\X)}{k!}$. Let us now compute $P_{\ba} = p(x_{1} = \alpha_{1}, x_{2} = \alpha_{2}, \ldots, x_{m} = \alpha_{m})$.
\begin{eqnarray}
  \label{eq:marginal-prob-proj}
  P_{\ba} &= \sum_{x_{m+1} \ldots x_{k}} p \left(\mathbf{x} = [ \alpha_{1},\ldots,\alpha_{k}, x_{m+1}, \ldots, x_{k}]  \right) \\
       &= \frac{1}{k!} \sum_{x_{m+1} \ldots x_{m}} \det \bL_{\left\{ \ba,x_{m+1},\ldots,x_{k}  \right\}}
\end{eqnarray}

Note that since the determinant equals 0 if there are repeated elements, it does not matter if we include repeated elements in the sum. Applying eq. (\ref{eq:det-bordered}), we obtain:
\begin{equation}
  \label{eq:mar-prob-proj-2}
  P_{\ba} = \frac{1}{k!} \sum_{ \mathbf{z}, x_{k}} \det \bL_{\left\{ \ba, \mathbf{z}  \right\}} \left(\bL_{x_{k}} - \bL_{x_{k},\left\{ \ba,\mathbf{z} \right\}} \bL_{\left\{ \ba,\mathbf{z} \right\}}^{-1}\bL_{\left\{ \ba,\mathbf{z}  \right\},x_{k}} \right) 
\end{equation}
where we have replaced $x_{m+1} \ldots x_{k-1}$ with a vector $\mathbf{z}$ of length $k-m-1$.
Next, we sum over $x_{k}$, applying eq. (\ref{eq:removing-item-proj}):
\begin{align}
  \label{eq:mar-prob-proj-2}
  P_{\ba} &= \frac{1}{k!} \sum_{ \mathbf{z}} \det \bL_{\left\{ \ba, \mathbf{z}  \right\}} \sum_{x_{k}} \left(\bL_{x_{k}} - \bL_{x_{k},\left\{ \ba,\mathbf{z} \right\}} \bL_{\left\{ \ba,\mathbf{z} \right\}}^{-1}\bL_{\left\{ \ba,\mathbf{z}  \right\},x_{k}} \right) \\
          &= \frac{1}{k!} \sum_{ \mathbf{z}} \det \bL_{\left\{ \ba, \mathbf{z}  \right\}} (k - (k-1))
\end{align}
Doing this recursively for $x_{k-1}, x_{k-2}, \ldots$ up to $x_{m+1}$, we obtain:
\begin{align}
  \label{eq:mar-prob-proj-3}
  P_{\ba} &= \frac{1}{k!} \left(  \det \bL_{\ba} \right) (k-m)(k-m-1) \ldots 1 \\
          &= \frac{(k-m)!}{k!} \det \bL_{\ba} 
\end{align}
which in turns implies: 
\begin{equation}
  \label{eq:marginal-ker-proj}
p(\ba \subseteq \X) = \det \bL_{\ba} 
\end{equation}

\subsection{Reduction to diagonal DPPs}
\label{sec:reduction_diag_proof}

Since  $k$-DPP are mixtures of diagonal $k$-DPPs we can write 
\begin{eqnarray}
  \label{eq:incl-prob-mixture}
p(  \ba \subseteq \X \big|k)= E[ p( \ba\subseteq \X)_\Y ]
\end{eqnarray}
where the outer expectation is over diagonal $k$-DPPs $\Y$, and
\begin{equation}
  \label{eq:incl-prob-proj}
  p( \ba\subseteq \X)_\Y = \det \left(  L(\Y)_{\ba} \right)
\end{equation}
with $L(\Y) = \mathbf{U}_{\Y,:}\mathbf{U}_{:,\Y}^\top$.

Let $\bY$ be a diagonal matrix with $y_{ii} = 1 $ if $i\in \Y$, and 0 otherwise. Then we may express the marginal probability of inclusion as:
\begin{align*}
  p(\ba \subseteq \X ) &=   E[\det \left( (\bU \vY \bU^\top)  _{\ba} \right) ] \\
                               &= E[\det \left( (\bU_{\ba,:} \vY \bU^\top_{:,\ba})  \right) ] \numberthis  \label{eq:marginal-prob-cb}
\end{align*}
where the expectation is over $\Y$. The determinant inside the expectation can be computed using the Cauchy-Binet theorem, giving:
\begin{align*}
p \left(\ba \subseteq \X \right) &= E[\sum_{\bb\slash |\bb|=|\ba|} \det \bU_{\ba \bb} \det (\vY \bU^\top)_{\bb \ba }]\\
&= E[\sum_{\bb\slash |\bb|=|\ba|} \det \bU_{\ba \bb}  \bU_{\ba \bb}^\top\prod_{i\in \bb} y_i] \\
                                 &= \sum_{\bb\slash |\bb|=|\ba|} p(\bb \subseteq \Y) \det \bU_{\ba \bb}  \bU_{\ba \bb}^\top \numberthis                                       \label{eq:inclusion-prob-general}
\end{align*}
In particular, for singletons $|\ba| = 1$ we recover  the inclusion probability of order 1 given in sec. \ref{sec:inclusion_kDPPs}.

Suppose now we have to measure the total variation distance between an inclusion probability of a $k$-DPP ($\pi$) and a DPP  approximation of it ($\tilde{\pi}$). Recalling that each is a mixture of diagonal DPPs with inclusion measure $\rho$ and $\tilde{\rho}$, we write
\begin{eqnarray*}
D_{m}(\pi, \widetilde{\pi} ) &=&
\frac{1}{{k \choose m}}\sum_{\ba} \Big| \pi(\ba) - \widetilde{\pi}(\ba) \Big)\Big| \\
&=& \frac{1}{{k \choose m}}\sum_{\ba} \Big{|}  \sum_{\bb} \det \bU_{\ba \bb}  \bU_{\ba \bb}^\top \Big(  \rho(\bb) - \widetilde{\rho}(\bb) \Big) \Big| \\
&\leq& \frac{1}{{k \choose m}}\sum_{\bb} \Big{|}  \rho(\bb) - \widetilde{\rho}(\bb) \Big|
 \sum_{\ba} \det \bU_{\ba \bb}  \bU_{\ba \bb}^\top 
 \end{eqnarray*}
 But $\sum_{\ba} \det \bU_{\ba \bb}  \bU_{\ba \bb}^\top = \sum_{\ba} \det (\bU^\top)_{ \bb\ba} \bU_{\ba \bb} =\det(\bU^\top \bU)_{\bb \bb}=1$. 
Thus 
$D_{m}(\pi, \widetilde{\pi} ) \leq D_{m}(\rho, \widetilde{\rho} )$, proving Lemma \ref{lm:diag-reduction}.

%Going back to DPPs, in a limit where $n$ grows and $k$ is fixed we expect the distribution of $|X|$ to tend to a Poisson distribution, rather than a Gaussian one 

\subsection{Computing an asymptotic expansion for diagonal inclusion probabilities}
\label{sec:asymptotics-inclusion-proof}

To derive the $O(1)$ and $O(n^{-1})$ terms in the inclusion probabilities, we take the exact expression (eq. \ref{eq:inclusion-prob-diag}) and inject the saddlepoint approximation (eq. (\ref{ eq:ESP-saddlepoint}), which yields: 

\begin{equation}
  p_{k}( \prod_{j \in \ba} z_{j}=1 ) =  \left(\prod_{i \in \ba} \frac{\lambda_{i}}{1+\lambda_{i}}\right) \frac{\sqrt{\psi''(\nus)}}{\sqrt{\psi''(\nus_{\ba}) - \psi_{\ba}''(\nus_{\ba})}}
  \exp \left( \psi(\nus_{\ba}) - \psi(\nus) -  \psi_{\ba}(\nus_{\ba}) + k \nus - (k-m) \nus_{\ba} \right)
\label{eq:inclusion-saddle}
\end{equation}
where $\psi_{\ba} = \sum_{i\in\ba} \psi_{i} $, $\psi'_{\ba} = \sum_{i\in\ba} \psi'_{i} $ and so on. The relative error in this approximation is of order  $O(n^{-2})$ \footnote{The reason the relative error is $O(n^{-2})$ is that we take a ratio of $O(n^{-1})$ errors that are actually the same up to a $O(n^{-1})$ term. Intuitively, the relative errors in the saddlepoint approximation of $e_k(\bl)$ and $e_{k-m}(\bl_{-\ba})$ are almost the same, and thus most of the error cancels when we take the ratio.} and we neglect it from now on. To get the $O(1)$ and $O(n^{-1})$ terms, we use a perturbation approach, where we treat $\epsilon = n^{-1}$ as a (scalar) perturbation parameter. The reason we have to use a perturbation approach is for lack of an analytical expression for the saddlepoint parameter $\nus$. 
To do so, we split eq (\ref{eq:inclusion-saddle}) into three terms: 
$$ A = \left(\prod_{i \in \ba} \frac{\lambda_{i}}{1+\lambda_{i}}\right)$$
$$ B = \frac{\sqrt{\psi''(\nus)}}{\sqrt{\psi''(\nus_{\ba}) - \psi_{\ba}''(\nus_{\ba})}} $$
$$ C = \exp \left( \psi(\nus_{\ba}) - \psi(\nus) -  \psi_{\ba}(\nus_{\ba}) + k \nus - (k-m) \nus_{\ba} \right)  $$

We shall find series for $B$  and $C$ of the form $B = b_{0} +  \epsilon b_{1} + \epsilon^{2} b_{2} + \ldots$, $C =  \exp(c_{0} +  \epsilon c_{1} + \epsilon^{2} c_{2} + \ldots)$. We will see that here $b_{0} = 1$. 
These series can in turn be used to obtain approximations of order $\epsilon^{0}$ and $\epsilon^{1}$ to the product $ABC$, namely:
\begin{equation*}
p_{k}( \prod_{j \in \ba} z_{j}=1 ) = A ( \exp(c_{0}) \left( 1 + \epsilon (c_{1} + b_{1})  + O(\epsilon^{2}) \right)
\end{equation*}
We note  $\psib=\frac{1}{n}\psi$, $\psib_{\ba}=\frac{1}{m}\psi_{\ba}$, $r = \frac{k}{n}$.
To obtain our perturbation series, we begin with the perturbed solution to the saddlepoint equation $\nus_{\ba}$, defined by:
\begin{align*}
  \nus_{\ba} &= \argmin_{\nu} \psi(\nu) - \psi_{\ba}(\nu) - (k-m)\nu \\
             &= \argmin_{\nu} \psib(\nu) - r\nu + m \epsilon ( \nu -  \psib_{\ba}(\nu)) \\
               &= \argmin_{\nu} f(\nu, \epsilon) \numberthis   \label{eq:nu-perturbation}
\end{align*}

Define the ansatz $\nus(\epsilon) = \argmin f(\nu, \epsilon) = \nu_{0}+\epsilon \nu_{1} + \epsilon^{2} \nu^{2} + \ldots$. From the saddlepoint equation we obtain:
\begin{align*}
  \psib'(\nus) - r + m \epsilon ( 1 -  \psib'_{\ba}(\nus)) = 0
\end{align*}
At order $\epsilon^{0}$, the equation implies:
\begin{equation}
  \label{eq:nu0}
\psib'(\nu_{0}) - r = 0
\end{equation}
so that $\nu_{0}$ equals the saddlepoint of the unperturbed problem.
At order $\epsilon^{1}$, we obtain:
\begin{equation}
  \label{eq:nu1}
\psib''(\nu_{0})\nu_{1} - m\left(1-\psib'_{\ba}(\nu_0) \right) = 0
\end{equation}

Further orders are not needed for our purposes. 

% At order $\epsilon^{2}$, we obtain:
% \begin{equation}
%   \label{eq:nu2}
% \psib''(\nu_{0})\nu_{2}+ \frac{1}{2}\psib^{(3)}(\nu_{0})\nu_{1}^{2} - m\psib''(\nu_0)\nu_1  = 0
% \end{equation}

We are now ready to insert these equations back into \eqref{eq:inclusion-saddle}. We begin with the exponential part. 
\begin{align}
  \label{eq:Cpert}
  C(\epsilon) &= \exp \left( n \left( f(\nus_{\ba},\epsilon) - \psib(\nus) - r\nus \right)  \right) \\
    &= \exp \left( n \left( f(\nu_{0} + \epsilon \nu_{1} + \epsilon^{2} \nu_{2} + \ldots, \epsilon) - \psib(\nu_{0}) - r\nu_{0} \right)  \right)  
\end{align}
We proceed with a similar perturbation for $f(\nus(\epsilon),\epsilon)$, $f = f_{0}+\epsilon f_{1} + \epsilon^{2} f_{2} + \ldots $
\begin{align}
  \label{eq:fpert-zero}
  f_{0} &= \psib(\nu_0) - r\nu_{0} \\
  f_{1} &= \psib'(\nu_0)\nu_{1} - r\nu_{1} + m\left(\nu_{0} - \psib_{\ba}(\nu_{0})\right) \\
        &=  m\left(\nu_{0} - \psib_{\ba}(\nu_{0})\right) \\
  f_{2} &= \psib'(\nu_0)\nu_{2} - r\nu_{2} + \frac{1}{2} \psib''(\nu_{0})\nu_{1}^{2} + m \left( \nu_{1} - \psib_{\ba}'(\nu_{0})\nu_{1} \right) \\
          &= -\frac{\nu_{1}^{2}}{2}\psib''(\nu_{0}) \\
\end{align}
where we have made use of eq. \eqref{eq:nu0} and  \eqref{eq:nu1}.
Inserting (\ref{eq:fpert-zero}) into (\ref{eq:Cpert}), we find
\begin{equation}
  \label{saddlepoint-exponential-part}
C(\epsilon) = \exp \left( f_{1} + \epsilon f_{2} + O(\epsilon^{2}) \right)
  \end{equation}
  
We now proceed with the other factor of  \eqref{eq:inclusion-saddle}, $B$, involving a ratio of square roots:
\begin{equation}
  \label{eq:saddlepoint-sqrt-part}
  B(\epsilon) = \frac{\sqrt{\psi''(\nus)}}{\sqrt{\psi''(\nus_{\ba}) - \psi_{\ba}''(\nus_{\ba})}}
   = \frac{\sqrt{\psib''(\nu_{0})}}{\sqrt{\psib''(\nus_{\ba}) - m \epsilon \psib_{\ba}''(\nus_{\ba})}}
\end{equation}
Note that
\begin{align}
  \sqrt{\frac{a}{a-\epsilon}} = \sqrt{\frac{1}{1-\frac{\epsilon}{a}}} = \sqrt{\left( 1+\frac{\epsilon}{a}+O(\epsilon^{2}) \right)} = 1+\frac{\epsilon}{2a}+O(\epsilon^{2})
\end{align}
It is immediate from the above that $B=1 + O(\epsilon)$, and that therefore: 
\begin{align*}
  p_{k}( \prod_{j \in \ba} z_{j}=1 ) &= \left(\prod_{i \in \ba} \frac{\lambda_{i}}{1+\lambda_{i}}\right) \exp \left( f_{1} + O(\epsilon) \right) \\
  &= \left(\prod_{i \in \ba} \frac{\lambda_{i}}{1+\lambda_{i}}\right) \exp \left( m\left(\nu_{0} - \psib_{\ba}(\nu_{0})\right) + O(\epsilon) \right) \\
   &= \left( \prod_{i \in \ba} \frac{\lambda_{i} \exp(\nu_{0})}{1+\lambda_{i}\exp(\nu_{0})}  \right) \left( 1 + O\left(\frac{1}{n} \right) \right) \numberthis   \label{eq:inclusion-prob-limit}
\end{align*}

For numerical purposes it is interesting to obtain the $O(\epsilon)$ term, which requires the first-order approximation to $B$:
\begin{equation}
  \label{eq:sqrt-ratio-first-order}
  B(\epsilon) = 1 - \epsilon \frac{1}{ 2 \psib''(\nu_{0})}  \left( \psib^{(3)} ( \nu_{0}) \nu_{1} - m \psib_{\ba}'' (\nu_{0}) \right) + O \left( \epsilon^{2} \right)
\end{equation}

This completes the proof of lemma \ref{lm:inclusion-probs}.

\subsection{An easy-to-compute correction to the $O(n^{-1})$ approximation}
\label{sec:corrected_approx}

One way to get an improved estimate of inclusion probabilities is to compute the $O(n^{-1})$ term in the saddlepoint expansion, and that is what we recommend for first-order inclusion probabilities. It is harder to use when $m>1$, and in this section we describe a correction that is easy to compute and yields interesting insights into the approximation.
From lemma \ref{lm:sums-inclusion-prob}, we know what the sum of the inclusion measure for a $k$-DPP over all sets of size $m$ should equal ${ k \choose m}$, while for a DPP with marginal kernel $\bK$ it equals:
\begin{equation}
  \label{eq:sums-of-inclusions}
 \sum_{\ba, |\ba|=m} \det \bK_{\ba} = e_{m}(\bK) 
\end{equation}
the m-th ESP of matrix $\bK$.
In the matched DPP $\bK$ equals $\left( \bI + e^{\nu}\bL \right)^{-1}e^{\nu}\bL$, and the eigenvalues of $\bK$ are $\eta_{i} = \frac{e^{\nu}\lambda_{i}}{1+ e^{\nu} \lambda_{i}}$.
What eq. \eqref{eq:sums-of-inclusions} implies is that for the approximation to be exact at order $m$, we need to have $e_{m}(\be) = {k \choose m}$.
One shows easily that this is true if and only if $\be$ has exactly $k$ entries that equal 1, and the rest are all zero, which happens to be just the case described in result \ref{result:max-rank-dpp}.
\begin{lemma}
  Let $\be \in [0,1]^{n}$, with $\sum \eta_{i} = k$, and let $m < k$. Then ${ k \choose m} \leq e_{m}(\be) \leq { n \choose k} (\frac{k}{n})^{m}$ 
\end{lemma}
\begin{proof}
  We seek the extrema of $e_{m}(\be)$ under the equality constraint $\sum \eta_{i} = k$ and the inequality constraints $0 \leq \eta_{i} \leq 1$. Maximisation is easy $e_{m}(\be)$ is a concave function (as a consequence of Schur concavity, \citet{Sra:LogIneqESP}), and we have linear constraints, so that any maximum is unique. 
  The Lagrangian equals:
  \begin{equation}
    \label{eq:lagrangian}
   \mathcal{L}(\be,\nu,\bm{\gamma}),\bm{\delta}) = e_{m}(\be) - \nu \left( \sum \eta_{i} - k \right) - \bm{\gamma}^{\top}  \be  - \bm{\delta}^{\top} \left( \be - 1 \right)
  \end{equation}
  and the Karush-Kuhn-Tucker conditions imply that for all $i$:
  \begin{align}
    \label{eq:KKT}
    \frac{\partial}{\partial \eta_{i}} e_{m}(\be) = e_{m-1}(\be_{-i}) =  \nu + \gamma_{i} + \delta_{i} \\ 
    \gamma_{i} \eta_{i} = 0  \\ 
    \delta_{i} (\eta_{i}-1) = 0  \\
  \end{align}
  The solution where $\eta_{i} = \frac{k}{n}$ for all $n$ only has inactive constraints, and is a maximum.
  Finding minima requires a bit more work. Let us consider a potential solution $\be$, and split into three consecutive parts: the zero values (active constraints under the $\gamma$ multiplier), the values contained above zero and below one (inactive constraints), and the values equal to one. The KKT conditions imply that for all $j$ such that the $j$-th constraint is inactive, $e_{m-1}(\eta_{-j}) = \nu$, meaning that removing any $0 < \eta_{j} < 1$ has the same effect, which implies that all these values are the same. Consequently, we can reparametrise the solution as
  \[ \be = [ 0, 0, \ldots, 0, a, a, \ldots, a, 1, 1, \ldots, 1 ]\]
  where $0 < a < 1$. Since $e_{m}$ is invariant to permutations there is no loss of generality. 
  Next, notice that $e_{m}([0 \bb]) = e_{m}(\bb)$ for all $m$. Further:
  \begin{equation}
    \label{eq:esp-concatenate-one}
    e_{m}([\bb 1]) = e_{m+1}(\bb)+e_{m}(\bb)
  \end{equation}
so that $e_{m}([ 0, 0, \ldots, 0, a, a, \ldots, a, 1, 1, \ldots, 1 ])$ is just a weighted sum of elementary symmetric polynomials of the vector $[a, a, \ldots, a]$, so that $a$ needs to be as small as possible under the constraints. The minima must therefore all have $k$ values equal to one, and the rest zero. Evaluating $e_{m}$ at the two extrema yields the bound. 
\end{proof}

The least favorable case is therefore when $\bL$ has a flat spectrum (which thankfully should not happen), but even then asymptotic equivalence holds: one can verify that $ { n \choose k} (\frac{k}{n})^{m} \asymp { k  \choose m }$. However, at finite orders, one can improve the approximation by making sure it sums to the right quantity: i.e., approximate inclusion probabilities via:
\begin{equation}
  \label{eq:est-improved}
\tilde{\pi}_{corrected}(\ba) = \frac{{ k  \choose m }}{e_{m}(\be)} \det \bK_{\alpha}
\end{equation}
When $m=1$ the correction does nothing (the correction factor equals 1), but at higher orders we have found that it can sometimes reduce relative error by a factor of 10.

\end{document}